\documentclass[11pt,reqno]{amsart}
\usepackage{amssymb}
\usepackage{amsfonts}
\usepackage{mathrsfs}
\usepackage{amsmath}
\usepackage{graphicx}
\usepackage{hyperref}
\usepackage{float}
\usepackage{epstopdf}
\usepackage{color}
\usepackage{bm}
\usepackage{comment}
\usepackage{soul}

\allowdisplaybreaks
\newtheorem{theorem}{Theorem}[section]
\newtheorem{proposition}[theorem]{Proposition}
\newtheorem{lemma}[theorem]{Lemma}

\newtheorem{definition}[theorem]{Definition}

\newtheorem{thm}{Theorem}

\def\G{\mathcal{G}}
\def\mcD{\mathcal{D}}
\def\mcE{\mathcal{E}}
\def\mcF{\mathcal{F}}
\def\mcM{\mathcal{M}}

\def\mcR{\mathcal{R}}

\numberwithin{equation}{section}

\begin{document}
\title[Brownian motion on the golden ratio Sierpinski gasket]{Brownian motion on the golden ratio Sierpinski gasket}

\author{Shiping Cao}
\address{Department of Mathematics, Cornell University, Ithaca 14853, USA}
\email{sc2873@cornell.edu}
\thanks{}

\author{Hua Qiu}
\address{Department of Mathematics, Nanjing University, Nanjing 210093, China}
\email{huaqiu@nju.edu.cn}
\thanks{The research of Qiu was supported by the National Natural Science Foundation of China, Grant 12071213.}

\subjclass[2010]{Primary 28A80}

\date{}

\keywords{golden ratio Sierpinski gasket, infinite graph, Dirichlet forms, heat kernel estimates}

\begin{abstract}
 We construct a strongly local regular Dirichlet form on the golden ratio Sierpinski gasket, which is a self-similar set without any finitely ramified cell structure, via a study on the trace of electrical networks on an infinite graph. The Dirichlet form is self-similar in the sense of an infinite iterated function system, and is decimation invariant with respect to a graph-directed construction. A theorem of uniqueness is also provided. Lastly, the associated process satisfies the two-sided sub-Gaussian heat kernel estimate. 
\end{abstract}
\maketitle

\section{introduction}
The golden ratio Sierpinski gasket $\G$ is a typical example of self-similar sets satisfying the finite type property \cite{BR}, which arises in the study of the Hausdorff dimension of self-similar sets with overlaps \cite{LN,NW,RW}. Let $q_0=(\frac{1}{2},\frac{\sqrt{3}}{2}),q_1=(0,0),q_2=(1,0)$ be the three vertices of an equilateral triangle in $\mathbb{R}^2$, and
\[
\begin{aligned}
F_0(x)=&\rho^2(x-q_0)+q_0,\\
F_1(x)=\rho(x-q_1)&+q_1,\quad F_2(x)=\rho(x-q_2)+q_2,
\end{aligned}
\]
with $\rho=\frac{\sqrt{5}-1}{2}$ being the golden ratio. The gasket $\G$ is the invariant set associated with the iterated function system (i.f.s. for short) $\{F_0,F_1,F_2\}$. See Figure \ref{goldsg}. 

\begin{figure}[h]
	\centering
	\includegraphics[width=5.5cm]{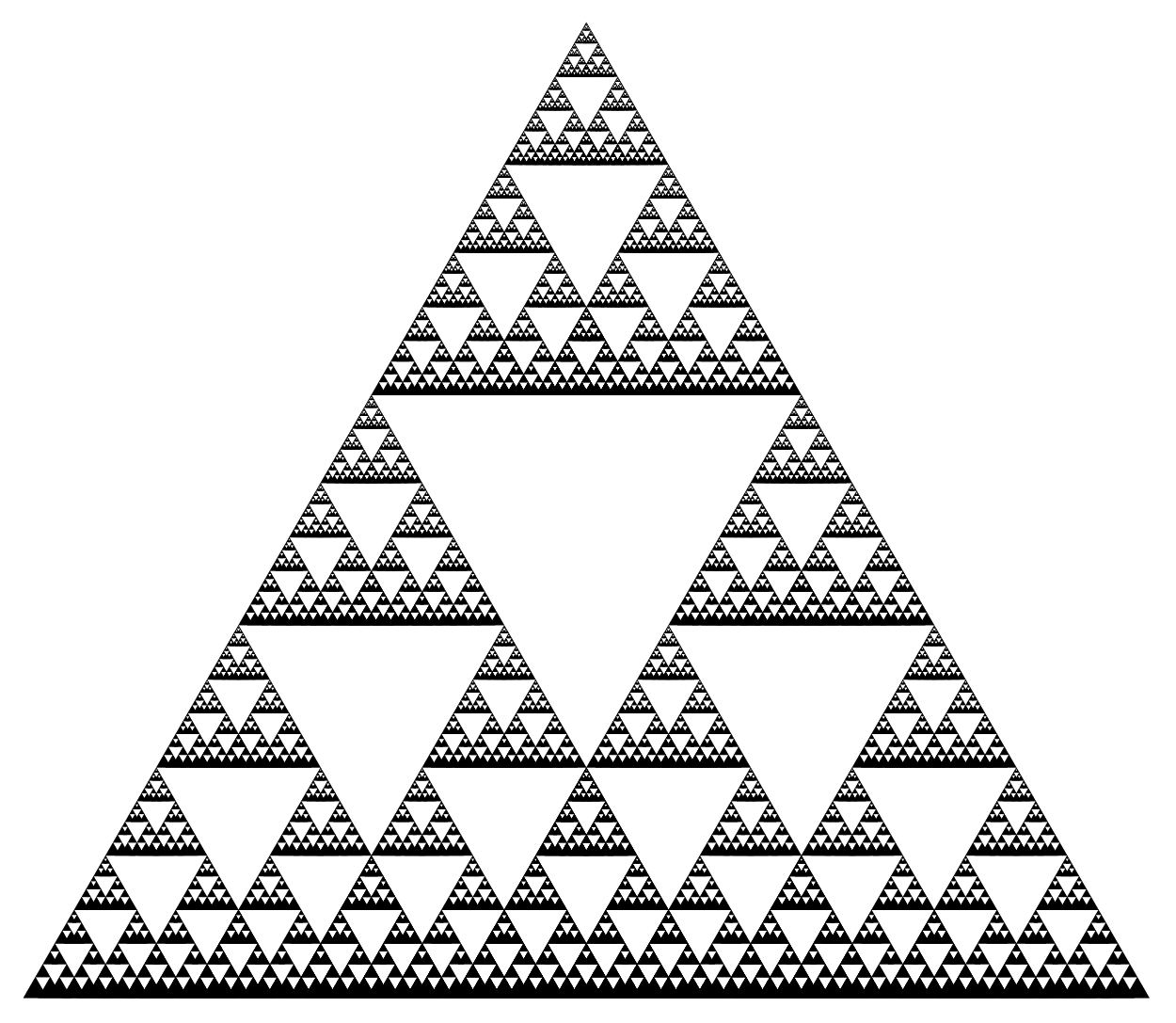}
	\begin{center}
		\caption{The golden ratio Sierpinski gasket $\G$.}\label{goldsg}
	\end{center}
\end{figure}

The large overlap $F_1\G\cap F_2\G$ makes $\G$ different from the existing examples of self-similar sets on which Brownian motions are constructed. 

First, any effort to disconnect the bottom line of $\G$ requires the removal of infinitely many points, so there is not a finitely ramified cell structure \cite{T} on $\G$. Well-known classes of fractals with finitely ramified cell structures include Lindstr$\o$m's nested fractals \cite{Lindstrom}, Kigami's post-critically finite (p.c.f.) self-similar sets \cite{ki1,ki2}, finitely ramified graph-directed fractals \cite{CQ,HN}, and some Julia sets of polynomials \cite{ADS,FS,RT} or rational functions \cite{CHQES}. See \cite{BP,G,kus} for pioneering works on the Sierpinski gasket, and also books \cite{B,ki3} for systematic discussions.

Second, although there is a graph-directed construction related with $\G$ (see Section 2), by dividing $\G$ into blocks of nearly the same size, the graph will be much complicated. The deep and famous constructions on the Sierpinski carpet \cite{BB,BB1,BB2} by Barlow and Bass, and on certain symmetric fractals \cite{KZ} by Kusuoka and Zhou will be extremely difficult here. See also \cite{BBKT} for a theorem of uniqueness on the Sierpinski carpet.
 
Instead, thanks to the golden ratio, there is an `infinite cell structure' on $\G$. For the first level, we consider the cell $F_0\G$ and its images of combinations of $F_1,F_2$. The union of these cells covers $\G$ except the bottom line. For each such cell, we can find a finite word $w$, and a contraction map $F_w=F_{w_1}\circ F_{w_2}\circ \cdots \circ F_{w_m}$, so that the cell can be written as $F_w\G$. We name the collection of all such words $W_1$, and construct a  resistance form \cite{ki3} on $\G$, that is self-similar in the sense of the infinite i.f.s. $\{F_w\}_{w\in W_1}$. Roughly speaking, we have the following theorem, see Theorem \ref{thm55}, \ref{thm56} and \ref{thm58} for detailed and formal results.

\begin{thm}
There exists a unique strongly local regular resistance form $(\mcE,\mcF)$ on $\G$ such that $f\in\mcF$ if and only if $f\circ F_w\in \mcF$ for all $w\in W_1$ and $\sum_{w\in W_1} \rho_w^{-\theta}\mcE(f\circ F_w,f\circ F_w)<\infty$, where $\rho_w$ is the similarity ratio of $F_w$ and $0<\theta<1$ is a constant. In addition, 
	\[\mcE(f,f)=\sum_{w\in W_1} \rho_w^{-\theta}\mcE(f\circ F_w,f\circ F_w).\]
Moreover, the form is decimation invariant with respect to the graph-directed construction of $\G$.
\end{thm}

The form $(\mcE,\mcF)$ is then a strongly local regular Dirichlet form on $L^2(\G,\mu_H)$, where $\mu_H$ is the normalized Hausdorff measure on $\G$. There is a diffusion process associated by the well-known theorem \cite{FOT}. Although, our construction is based on an infinite i.f.s., the behavior of the process is same on each cell before hitting the boundary, up to a time scaling, since any cell can be decomposed in a same manner. 

In addition, by following the well-established method of Hambly and Kumagai \cite{HK}, which is organized in Barlow's book \cite{B}, we can obtain a sub-Gaussian heat kernel estimate (see Section 7). We refer to \cite{BP,FHK,kum} for earlier results on transition density estimates on fractals. 

\begin{thm}
	There is a symmetric transition density $p(t,x,y)$ associated with the form $(\mcE,\mcF)$ on $\G$. In addition, there are constants $c_1,c_2,c_3,c_4$  so that
	\[c_1t^{-d_H/\beta}\exp\big(-c_2(\frac{d(x,y)^\beta}{t})^{\frac{1}{\beta-1}}\big)\leq p(t,x,y)\leq c_3t^{-d_H/\beta}\exp\big(-c_4(\frac{d(x,y)^\beta}{t})^{\frac{1}{\beta-1}}\big),\]
	for $0<t\leq 1$, with $\beta=\theta+d_H$, where $d_H\approx1.6824$ is the Hausdorff dimension of $\G$.  
\end{thm}

 The main effort is to get an estimate of the resistance metric $R$ on $\G$ as $c_1d(x,y)^\theta\leq R(x,y)\leq c_2d(x,y)^\theta$ for some constant $c_1,c_2>0$.

We organize the structure of the paper as follows. In Section 2, we will briefly introduce some facts about the geometry of $\G$. From Section 3 to 5, we study the trace of forms on an infinite graph. In Section 3, we establish the resistance forms on the graph. In Section 4, we study the trace map and a related renormalization map. We will show the jointly continuity of the renormalization map. In Section 5, we show that there is a unique solution to a renormalization problem. With all these preparations, we construct the resistance form on $\G$ in Section 6, and at the same time we get an upper bound estimate of the resistance metric. Lastly, we obtain the transition density estimate through a lower bound estimate of the resistance metric in Section 7.

Before ending this section, we remark that the result in this paper has a natural extension, by replacing $0<\rho<1$ to be a real root of $x^n-2x+1$ with $n\geq 4$, and taking the contraction ratios corresponding to $F_0, F_1, F_2$ to be $1-\rho, \rho,\rho$. Indeed, we will obtain a class of gaskets that possess a similar overlapping structure of $\G$, see Figure \ref{pisot}. 
\begin{figure}[htp]
	\includegraphics[width=5.5cm]{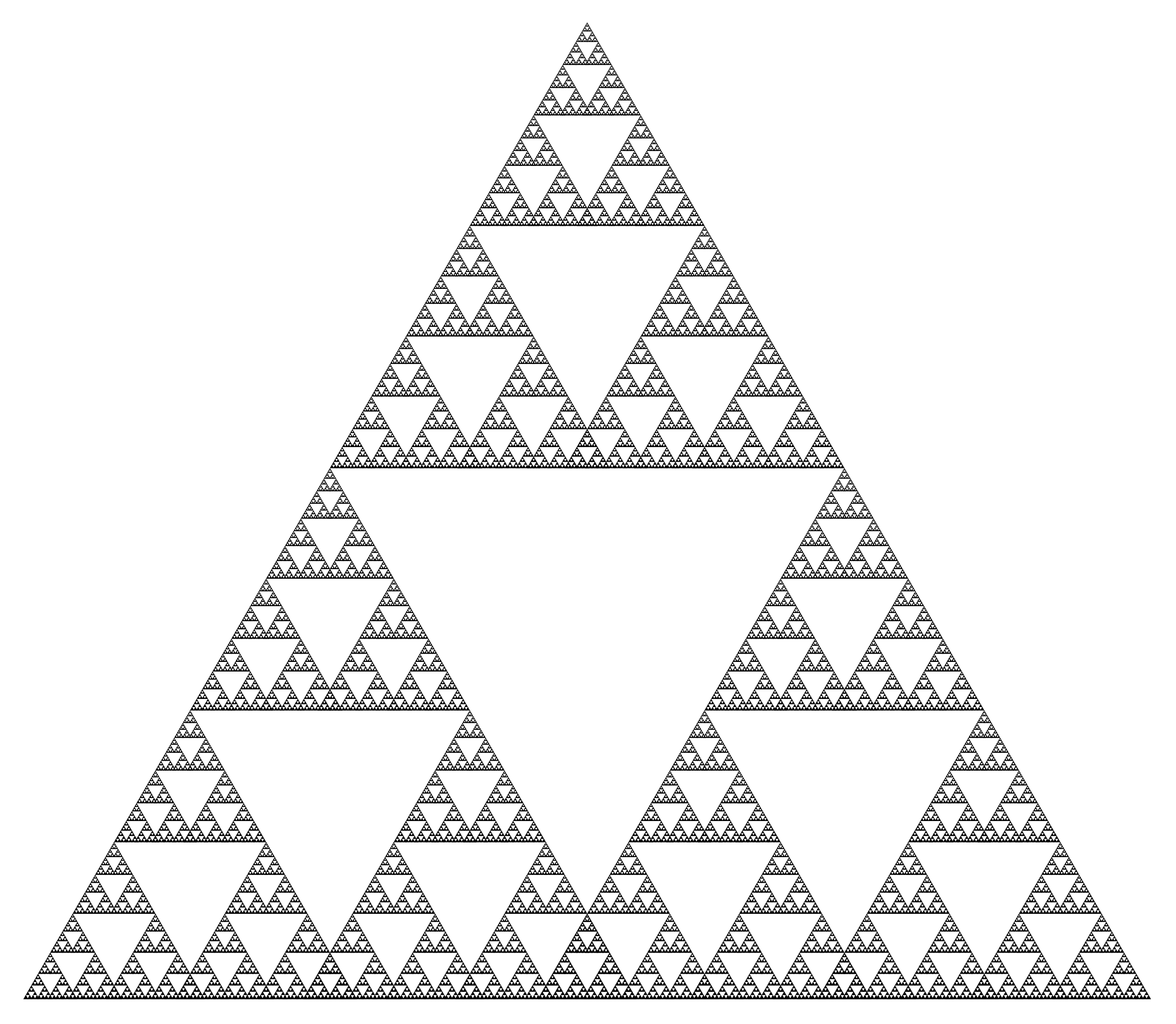}
	\caption{A gasket with $0<\rho<1$ being a root of $x^4-2x+1=0$.}\label{pisot}
\end{figure}

\section{Preliminary}
The golden ratio Sierpinski gasket $\G$ is one of the typical examples of self-similar sets with overlaps but satisfying the finite type property. 

Let $K$ be a general self-similar set associated with an i.f.s. $\{F_i\}_{i=0}^{N-1}$ with contraction ratios $\{\rho_i\}_{i=0}^{N-1}$. For $m\geq 1$, we call $w=w_1w_2\cdots w_m$ with $0\leq w_i<N$, a \textit{word} of length $m$(denoted by $|w|$), and call $\emptyset$ the empty word. We denote the set of all words by $\tilde{W}_*$. For any word $w\in \tilde{W}_*$, we write $F_w=F_{w_1}\circ F_{w_2}\cdots\circ F_{w_{|w|}}$, and let $F_\emptyset$ be the identity map for consistency. Let $\rho_*=\min\{\rho_i: 0\leq i< N\}$.

\begin{definition}[finite type property]\label{lemma11}
A self-similar set $K$ is of \emph{finite type} if there are only finite many maps $h=F_w^{-1}F_v$ with $w,v\in \tilde W_*$ and $F_w K\cap F_v K\neq \emptyset$, and with similarity ratio $\rho_h\in(\rho_*,1/\rho_*)$. 
\end{definition}

The finite type property of $K$, formulated in algebraic terms, was introduced in \cite{BR} by Bandt and Rao. It guarantees the existence of an `almost non-overlapping' graph-directed construction (see \cite{BR,NW} for details) of $K$, which is quite useful for calculating the Hausdorff dimension of $K$. See \cite{LN,RW} for more flexible variants of the finite type property. 

It is easy to verify that $\G$ satisfies the finite type property, noticing that $F_{122}\G=F_{211}\G$. In particular, it has the following \textit{graph-directed construction} \cite{MW}. 

\begin{definition}[a graph-directed construction of $\G$]\label{def12}
(a). Let $K_1=\G$ and $K_2=\overline{\G\setminus F_{22}\G}$. 

(b). Let $\Gamma(S,E)$ be a directed graph with the vertex set $S=\{1,2\}$, and the edge set $E=\{e_i\}_{i=1}^6$, where $e_1=(1,2),e_2=(1,1),e_3=(2,1),e_4=(2,2),e_5=(2,2),e_6=(2,1)$. 

(c). Define $\psi_{e_1}=Id$, $\psi_{e_2}=F_{22}$, $\psi_{e_3}=F_0$, $\psi_{e_4}=F_1$, $\psi_{e_5}=F_{21}$, $\psi_{e_6}=F_{20}$.      
\end{definition}

Clearly, we have 
\[K_1=\bigcup_{i=1}^2 \psi_{e_i}K_{e_{i,2}}\text{ and }K_2=\bigcup_{i=3}^6 \psi_{e_i}K_{e_{i,2}},\]
where we use the notation $e_i=(e_{i,1},e_{i,2})$ for a directed edge. In addition, there exist bounded open sets $O_1$ and $O_2$ such that $\bigcup_{i=1}^2 \psi_{e_i}O_{e_{i,2}}\subset O_1$ and $\bigcup_{i=3}^6 \psi_{e_i}O_{e_{i,2}}\subset O_2$, where the unions are disjoint. See Figure \ref{graph-directed} for an illustration. 

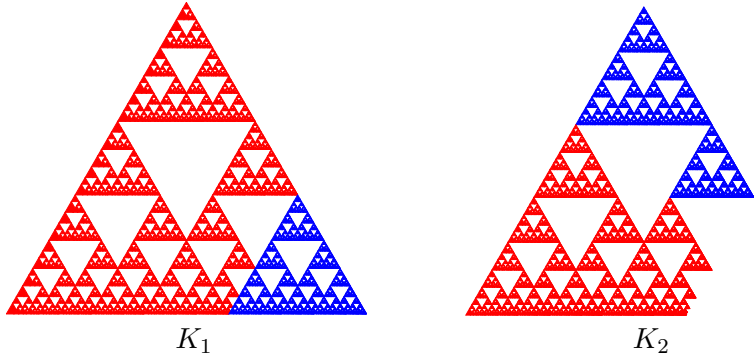
\begin{figure}[h]
	\centering
	\includegraphics[width=5cm]{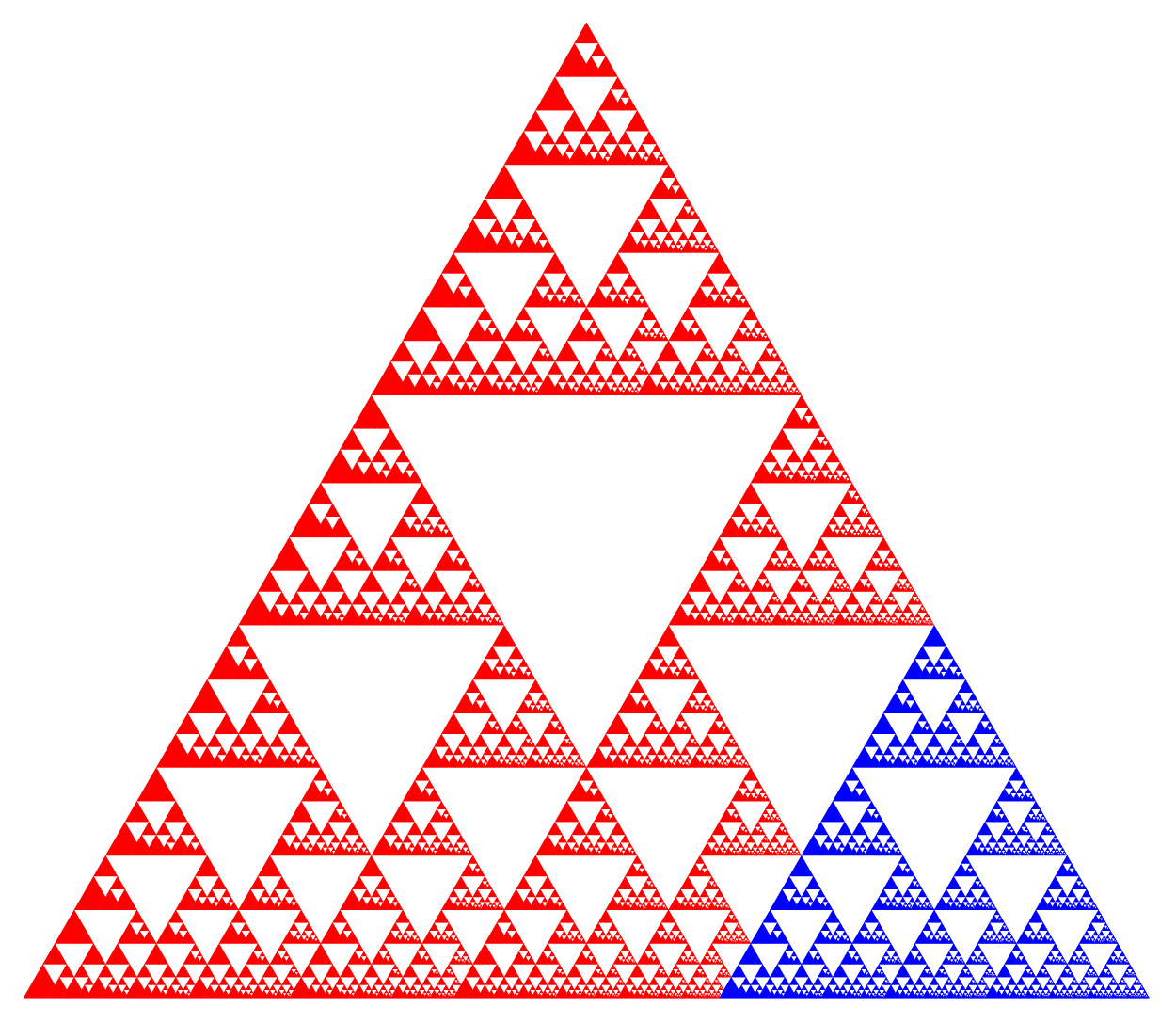}\hspace{1cm}
	\includegraphics[width=4cm]{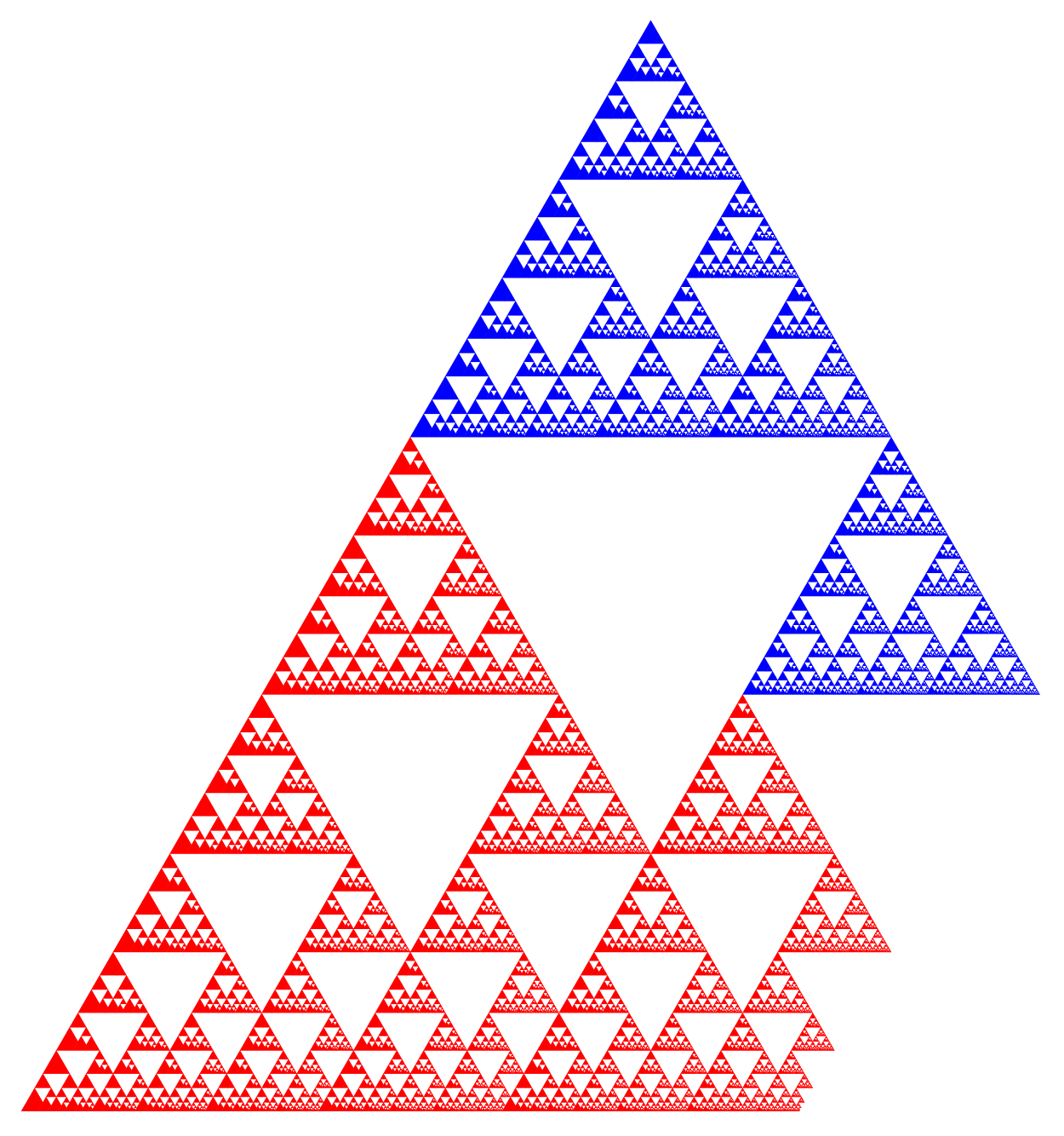}
	\begin{picture}(0,0)
	\put(-225,-10){$K_1$}\put(-52,-10){$K_2$}
	\end{picture}
	\begin{center}
		\caption{A graph-directed construction of $\G$.}\label{graph-directed}
	\end{center}
\end{figure}

Then similar to the open set condition situation, one can calculate the exact value of the Hausdorff dimension of $\G$ to be \[d_H=\frac{\log\eta}{-2\log\rho}\approx1.6824\] with $\eta$ being the largest root of $x^3-6x^2+5x-1$. In addition, the associated Hausdorff measure of $\G$ is positive and finite. See details in \cite{NW} by Ngai and Wang.

In this paper, we take $\mu_H$ to be the normalized Hausdorff measure on $\G$, i.e. $\mu_H(\G)=1.$ 
 It is not hard to verify that $\mu_H$ is volume doubling. 
 
\begin{lemma}\label{lemma14}
	Let $B_s(p)=\{q\in \G:d(p,q)<s\}$. There are constants $c_1,c_2>0$ such that  
		\[c_1s^{d_H}\leq\mu_H(B_s(p))\leq c_2s^{d_H}, \quad\forall p\in \G,0<s\leq1.\]
\end{lemma}

There is a \textit{geodesic metric} $d_g$ on $\G$ equivalent to the Euclidean metric $d$. 
\begin{lemma}\label{lemma15}
	For $p,q\in\G$, let $d_g(p,q)=\inf\{|\gamma|:\gamma \text{ is a path connecting }p,q, \text{ and }\gamma\subset\G\}$. Then there exists a constant $c\geq 1$ such that 
	\[d(p,q)\leq d_g(p,q)\leq cd(p,q),\quad \forall p,q\in\G.\]
\end{lemma}

The proof relies on the finite type property. The rough idea is to link $p,q$ with a bounded number of cells of diameter approximating to $d(p,q)$. 

Clearly, there is always a path admitting the infimum length between $p,q$. So the metric space $(\G,d_g)$ satisfies the so-called \textit{midpoint property}, i.e. for any $p,q\in \G$, there exists $p'$ so that $d_g(p,p')=d_g(p',q)=\frac{1}{2}d_g(p,q)$. The space $(\G,d_g,\mu_H)$ is then a \textit{fractional metric space}, see  \cite{B}, Definition 3.2.

We will return to look at the geometric properties of $\G$ listed in this section. But first, from Section 3 to 5, we will instead consider an infinite i.f.s. and the associated infinite graph. 

\section{Resistance forms on the infinite graph $V_1$}
The golden ratio Sierpinski gasket $\G$ can be realized as an invariant set of an infinite i.f.s. For convenience, we introduce some notations. For any word $w,w'\in \tilde{W}_*$, we write $ww'$ for the concatenation of $w,w'$. For $w=w_1w_2\cdots w_m$ and $0\leq l\leq m$, we write $[w]_l=w_1w_2\cdots w_l$. The following notations will be a little different from the standard ones.\vspace{0.15cm}

\noindent\textbf{Notation.}
Choose a set of finite words $W_1\subset\bigcup_{n=0}^\infty \{1,2\}^n\times \{0\}$ so that

\noindent1. for any $w\in \bigcup_{n=0}^\infty \{1,2\}^n\times \{0\}$, there exists $w'\in W_1$ such that $F_w=F_{w'}$;

\noindent2. for different words $w,w'\in W_1$, we have $F_w\neq F_{w'}$. 

In addition, based on $W_1$, we introduce some more notations.

\noindent\textit{(a). For $n\geq 1$, define $W_{1,n}=\{w\in W_1:|w|=n\}$;}

\noindent\textit{(b). For $m\geq 2$, define $W_m:=W_1^m=\{w_1w_2\cdots w_m: w_i\in W_1, 1\leq i\leq m\}$;}

\noindent\textit{(c). Write $V_0=\{q_i\}_{i=0}^2$ and for $m\geq 1$, $V_m=\bigcup_{w\in W_m}F_wV_0$. Denote $\bar{V}_m$ the closure of $V_m$;} 

\noindent\textit{(d). For distinct $p,q\in V_1$, we denote $p\sim q$ if and only if $p,q\in F_wV_0$ for some $w\in W_1$, which induce an infinite graph $(V_1,\sim)$. See Figure \ref{V1} for an illustration.}\vspace{0.15cm}

\begin{figure}[h]
	\centering
	\includegraphics[width=6cm]{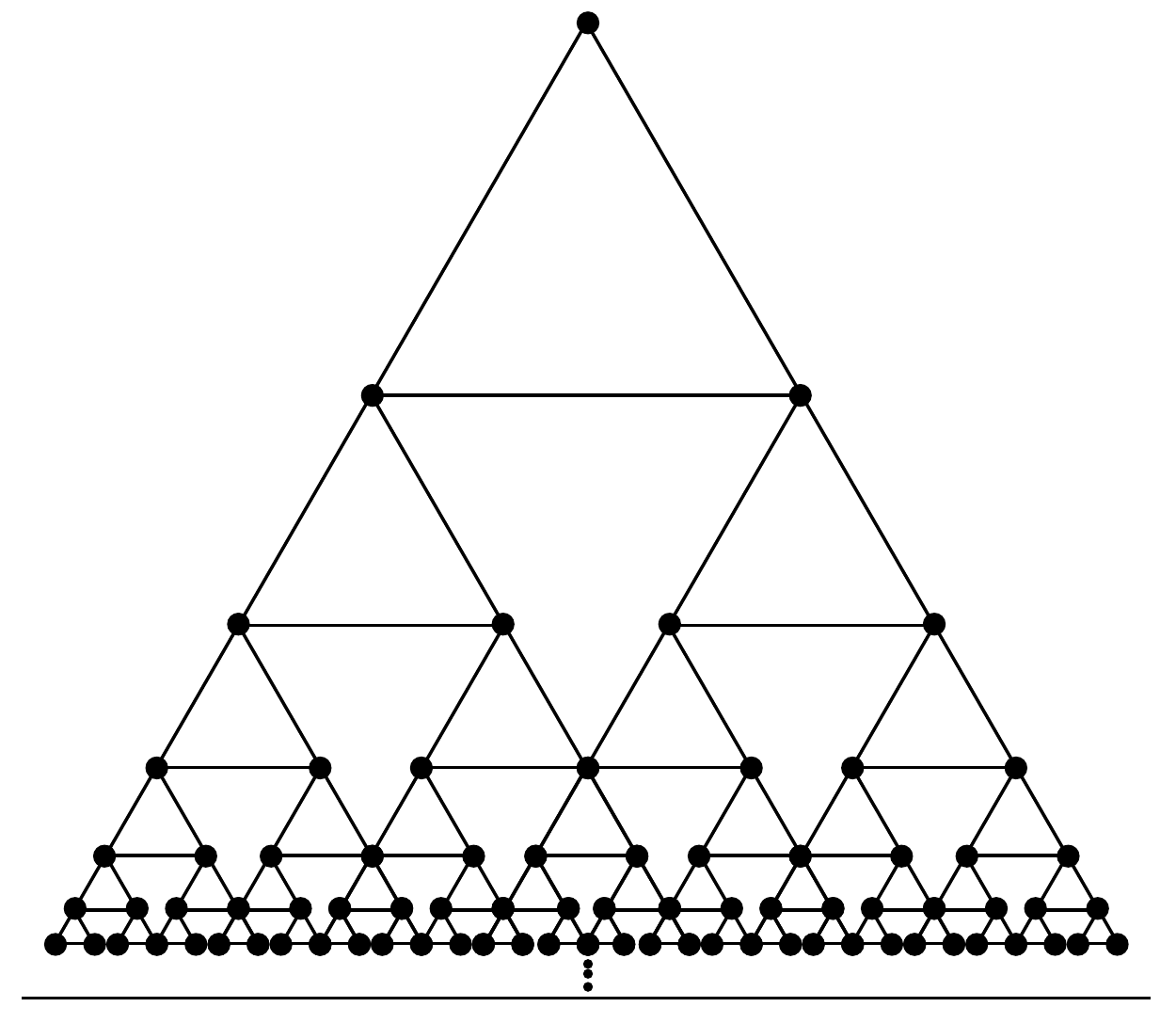}
	\caption{The infinite graph $(V_1,\sim)$. (The bottom line equals to $\bar V_1\setminus V_1$.)}\label{V1}
	\end{figure}

Obviously, we have $$\G=\overline{\bigcup_{w\in W_1}F_w\G},$$ and thus $\{F_w\}_{w\in W_1}$ is an infinite i.f.s associated with $\G$. See \cite{Mo} for more details about infinite i.f.s. The advantage of this i.f.s. lies in the fact that 
\[F_w\G\cap F_{w'}\G=F_wV_0\cap F_{w'}V_0, \quad\forall w\neq w'\in W_1.\]

In the rest of this section, we consider a class of resistance forms generated by decimation. For convenience of readers, we recall the general definition of resistance forms in the following. See \cite{ki3} for more details.

\begin{definition}\label{def21}
	Let $X$ be a set, and $l(X)$ be the space of all real-valued functions on $X$. A pair $(\mathcal{E},\mathcal F)$ is called a (non-degenerate) \emph{resistance form} on $X$ if it satisfies the following  conditions:
	
	\noindent (RF1). $\mathcal{F}$ is a linear subspace of $l(X)$ containing constants and $\mathcal{E}$ is a nonnegative symmetric quadratic form on $\mathcal F$; $\mathcal{E}(f):=\mathcal{E}(f,f)=0$ if and only if $f$ is constant on X.
	
	\noindent (RF2). Let `$\sim$' be an equivalent relation on $\mathcal{F}$ defined by $f\sim g$ if and only if $f-g$ is constant on X. Then $(\mathcal{F}/\sim, \mathcal{E})$ is a Hilbert space.
	
	\noindent (RF3). For any finite subset $V\subset X$ and  any $u\in l(V)$, there exists a function $f\in \mathcal{F}$ such that $f|_V=u$.
	
	\noindent (RF4). For any distinct $p,q\in X$, $R(p,q):=\sup\{\frac{|f(p)-f(q)|^2}{\mathcal{E}(f)}:f\in \mathcal{F},\mathcal{E}(f)>0\}$ is finite.
	
	\noindent (RF5). If $f\in \mathcal{F}$, then $\bar{f}={ \min\{\max\{f,0\}, 1\}}\in \mathcal{F}$ and $\mathcal{E}(\bar{f})\leq \mathcal{E}(f)$.
	\end{definition}
	
	Sometimes, we write $\mathcal F=Dom(\mathcal E)$, and abbreviate $(\mathcal E, \mathcal F)$ to $\mathcal E$ when no confusion occurs. It is well-known that $R(p,q)$ defined in (RF3) is a metric on $X$, named the \textit{effective resistance metric}. 

On the finite set $V_0$, a resistance form  $\mcD$ always has the form 
\begin{equation}\label{eqn21}
\mcD(f,g)=\frac 1 2\sum_{i,j} a_{i,j}\big(f(q_i)-f(q_j)\big)\big(g(q_i)-g(q_j)\big),\quad \forall f,g\in l(V_0),
\end{equation}
where $a_{i,i}=0$ and the $3\times 3$ matrix $(a_{i,j})$ is positive, symmetric and irreducible. For convenience, we write $\mcM$ for the collection of all resistance forms on $V_0$.  We view $\mcM$ as a subset of $\mathbb{R}^3$, which is not closed with induced topology.

Given a resistance form $\mcD$, we define a resistance form on $V_1$ associated with $\mcD$ in a self-similar manner, respecting the infinite i.f.s. $\{F_w\}_{w\in W_1}$.

\begin{definition}\label{def22}
	For $r>0$, $\mcD\in\mcM$, we define $\Psi_r\mcD$ as 
	\[\Psi_r\mcD(f,g)=\sum_{w\in W_1}r^{-|w|+1}\mcD(f\circ F_w,g\circ F_w),\]
	with $Dom(\Psi_r\mcD)=\{f\in l(V_1): \Psi_r\mcD(f)<\infty\}$.
\end{definition}

It is not hard to show that $\Big(\Psi_\lambda\mcD,Dom(\Psi_r\mcD) \Big)$ is a resistance form on $V_1$. However, to get a good resistance form, we need to restrict the range of $r$. 

\begin{proposition}\label{prop23}
Let $\mcD\in \mcM$ and $r<1$, then $Dom(\Psi_r\mcD)\subset C(\bar{V}_1)$ by a natural identification. In addition, if $\rho<r<1$, then $\Big(\Psi_r\mcD,Dom(\Psi_r\mcD) \Big)$ is a resistance form on $\bar{V}_1$, with the associated resistance metric $R(p,q)$ satisfying the estimate 
\begin{equation}\label{eqn22}
R(p,q)\leq \frac{4}{r^2(1-r)}R_0^*(\mcD)d(p,q)^{\frac{\log r}{\log \rho}}, \quad \forall p,q\in \bar{V}_1,
\end{equation}
where $R_0^*(\mcD)=\max_{p,q\in V_0}R_0(p,q)$ with $R_0$ being the resistance metric on $V_0$ associated with $\mcD$.
\end{proposition}
\begin{proof}
Obviously, $R(p,q)\leq R_0^*(\mcD)r^{n-1}$ for any distinct $p,q\in F_wV_0$ with $w\in W_{1,n}$ and $n\geq 1$. For $w\in\{1,2\}^n$, write $[w]_l=w_1w_2\cdots w_l$ and  $p_l=F_{[w]_l}(q_0)$, with $0\leq l\leq n$, then  
\[
R(p_i,p_j)\leq \sum_{l=i}^{j-1} R(p_l,p_{l+1})\leq R_0^*(\mcD)\sum_{l=i}^{j-1} r^l<R_0^*(\mcD)\frac{r^i}{1-r}, \quad \forall 0\leq i<j\leq n.
\]
In particular, this implies that $R(p,q)<\frac{2r^n}{1-r}R_0^*(\mcD)$ for any $p,q\in F_wV_1$ and $w\in \{1,2\}^n$. Now, if $p,q\in V_1$ and $d(p,q)<\rho^{n+1}$, then there exist $w,w'\in \{1,2\}^n$ such that $p\in F_{w}V_1, q\in F_{w'}V_1$, and $F_{w}V_1\cap F_{w'}V_1\neq\emptyset$, which implies that $R(p,q)\leq \frac{4r^n}{1-r}R_0^*(\mcD)$. As a consequence, we have 
\begin{equation}\label{eqn23}
R(p,q)\leq \frac{4}{r^2(1-r)}R_0^*(\mcD)d(p,q)^{\frac{\log r}{\log \rho}}, \quad \forall p,q\in V_1.
\end{equation}

On the other hand, for any $f\in Dom(\Psi_r\mcD)$, we immediately have 
\begin{equation}\label{eqn24}
|f(p)-f(q)|\leq \big(R(p,q)\Psi_r\mcD(f)\big)^{1/2}.
\end{equation}
Combining (\ref{eqn23}) and (\ref{eqn24}), we then get $Dom(\Psi_r\mcD)\subset C(\bar{V}_1)$.

To show the second assertion, we let $(X,R)$ be the completion of $(V_1,R)$, and recall Theorem 2.3.10 in \cite{ki3} to get that $\Big(\Psi_\lambda\mcD,Dom(\Psi_r\mcD) \Big)$ extends to be a resistance form on $X$. It suffices to show that the identity map $Id:V_1\to V_1$ extends to an homeomorphism from $(\bar{V}_1,d)$ to $(X,R)$ under the assumption $\rho<r<1$. First, by (\ref{eqn23}), $Id$ is continuous from $(V_1,d)$ to $(V_1,R)$. Next, let $f\in l(V_1)$ be a restriction of a linear function on $\mathbb{R}^2$. We have 
\[
\begin{aligned}
\Psi_r\mcD(f)&=\sum_{w\in W_1}r^{-|w|+1}\mcD(f\circ F_w)=\sum_{n=1}^\infty\sum_{w\in W_{1,n}}r^{-n+1}\mcD(f\circ F_w)\\
&=\sum_{n=1}^\infty \#W_{1,n}r^{-n+1}\rho^{2(n+1)}\mcD(f|_{V_0}),
\end{aligned}
\] 
where the last equality follows from the fact that $f$ is linear. Since $\#W_{1,n}\asymp\rho^{-n}$, we have $\Psi_r\mcD(f)<\infty$ when $r>\rho$, so that $f\in Dom(\Psi_r\mcD)$. Noticing that for any points $p\neq q\in \bar{V}_1$, we can find a linear function $f$ such that $f(p)\neq f(q)$, we have $Id$ is injective. Finally, due to the fact that $(\bar{V}_1,d)$ is a compact Hausdorff space and $(X,R)$ is the completion of $(V_1,R)$, we then have that $Id$ is an homeomorphism from $(\bar{V}_1,d)$ to $(X,R)$. This implies that $\Big(\Psi_r\mcD,Dom(\Psi_r\mcD) \Big)$ is a resistance form on $\bar{V}_1$, and (\ref{eqn22}) follows immediately from (\ref{eqn23}).
\end{proof}

\noindent\textbf{Remark.} The restriction $\rho<r<1$ is sharp. If $r\leq \rho$, there is no $f\in Dom(\Psi_r\mcD)$ such that $f(q_1)=0$ and $f(q_2)=1$. In fact, for any $f\in C(\bar{V}_1)$ with $f(q_1)=0$ and $f(q_2)=1$, the total  energy of $f$ on the union of the cells $F_wV_0,w\in W_{1,n}\cup W_{1,n+1}$ (noticing that this union will induce a connected subgraph in $(V_1,\sim)$) will be bounded  away from $0$ as $n\to\infty$. 

\section{A renormalization map}
In Proposition \ref{prop23}, we have shown that $\Psi_r\mcD$ extends to be a resistance form on $\bar{V}_1$ when $\rho<r<1$. It is natural to trace it back to $V_0$, noticing that $V_0\subset \bar V _1$.

\begin{definition}\label{def31}
Let $(\mcD_1,\mathcal{F}_1)$ be a resistance form on $\bar V_1$, we write
\[[\mcD_1]_{V_0}(u)=\inf\{\mcD_1(f):f|_{V_0}=u, f\in \mathcal{F}_1\}, \quad\forall u\in l(V_0).\]
Note that $[\mcD_1]_{V_0}$ is a resistance form on $V_0$ by the polarization identity.
For $\rho<r<1$ and $\mcD\in \mcM$, we define $\mcR_r\mcD=[\Psi_r\mcD]_{V_0}$, and call $\mcR_r$ the \emph{renormalization map}. Sometimes, we also write $\mcR(r,\mcD):=\mcR_r(\mcD)$. 
\end{definition}

The main purpose of this section is to show the continuity of  the map $\mcR(r,\mcD)$.

\begin{theorem}\label{thm32}
The map $\mcR(r,\mcD)$  is jointly continuous from $(\rho,1)\times \mcM$ to $\mcM$.
\end{theorem}

To prove Theorem \ref{thm32}, we need a study on the regularity of the resistance form $\Psi_r\mcD$.

\begin{proposition}\label{prop33}
Let $\mcD\in \mcM$ and $\rho<r_1<r_2<1$. Then

(a). $Dom(\Psi_{r_1}\mcD)$ depends only on $r_1$, and we have $Dom(\Psi_{r_1}\mcD)\subset Dom(\Psi_{r_2}\mcD)$. 

(b). $Dom(\Psi_{r_1}\mcD)$ is dense in $Dom(\Psi_{r_2}\mcD)$ in the sense that for any $f\in Dom(\Psi_{r_2}\mcD)$ and $\varepsilon>0$, there exists $g\in Dom(\Psi_{r_1}\mcD)$ such that
\[\Psi_{r_2}\mcD(f-g)<\varepsilon,\text{ and }f|_{V_0}=g|_{V_0}.\]
Moreover, $Dom(\Psi_{r_1}\mcD)$ is dense in $C(V_1)$ so that the resistance form is regular.
\end{proposition}
\begin{proof}
(a) is obvious, we only need to  prove (b). Let $f\in Dom(\Psi_{r_2}\mcD)$, and choose $n$ large enough so that 
\begin{equation}\label{eqn31}
\sum_{l=n}^\infty \sum_{w\in W_{1,l}}r_2^{-l+1}\mcD(f\circ F_w)<\varepsilon.
\end{equation}
For convenience, we rename the vertices $\{F_wq_0\}_{w\in W_{1,n}}$ to be $\{p_i\}_{i=1}^{N}$ with $N=\# W_{1,n}$, so that for each $i$, $p_i$ is on the left of $p_{i+1}$. Then, it is not hard to see
\[\begin{aligned}
&r_2^{-n}\big(\sum_{i=1}^{N-1}\big(f(p_i)-f(p_{i+1})\big)^2+\big(f(q_1)-f(p_{1})\big)^2+\big(f(q_2)-f(p_N)\big)^2\big)\\
\leq &c_1\big(\sum_{l=n}^\infty \sum_{w\in W_{1,l}}r_2^{-l+1}\mcD(f\circ F_w)\big)< c_1\varepsilon, 
\end{aligned}\]
where $c_1$ is a constant depending on $\mcD$ and $r_2$, but not on $n$.

Write $x_i$ for the $x$-coordinate of $p_i$, so we have $0<x_1<x_2<\cdots<x_N<1$. We introduce a piecewise linear function $u$ on $\mathbb{R}^2$ such that 

\noindent 1. $u(x,y)$ depends only on $x$;

\noindent 2. $u(q_1)=f(q_1)$, $u(q_2)=f(q_2)$, and $u(p_i)=f(p_i)$, $1\leq i\leq N$;

\noindent 3. $u(x,0)$ is linear on each interval $(0,x_1)$, $(x_N,1)$ and $(x_i,x_{i+1})$, $1\leq i\leq N-1$.

We define $g\in l(V_1)$ as
\[g(p)=\begin{cases}
f(p),\text{ if }p\in \bigcup_{l=1}^{n-1}\bigcup_{w\in W_{1,l}}\{F_wq_0\},\\
u(p),\text{ if }p\in \bigcup_{l=n}^{\infty}\bigcup_{w\in W_{1,l}}\{F_wq_0\}.
\end{cases}\]
By a similar estimate as in Proposition \ref{prop23}, one can check that $g\in Dom(\Psi_{r_1}\mcD)$, and 
\[
\sum_{l=n}^\infty \sum_{w\in W_{1,l}}r_2^{-l+1}\mcD(g\circ F_w)
\leq c_2r_2^{-n}\big(\sum_{i=1}^{N-1}\big(f(p_i)-f(p_{i+1})\big)^2+\big(f(q_1)-f(p_{1})\big)^2+\big(f(q_2)-f(p_N)\big)^2\big),
\]
where $c_2$ depends only on $\mcD$ and $r_2$. So we have $\Psi_{r_2}\mcD(f-g)\leq c_3\varepsilon$ for some constant $c_3$. Since $\varepsilon$ can be arbitrarily small, we have that $Dom(\Psi_{r_1}\mcD)$ is dense in $Dom(\Psi_{r_2}\mcD)$. Finally, the claim that $Dom(\Psi_{r_1}\mcD)$ is dense in $C(\bar{V}_1)$ follows from a same argument. 
\end{proof}

\begin{proof}[Proof of Theorem \ref{thm32}]
	Let $r_n\to r\in (\rho,1)$ and $\mcD_n\to \mcD\in \mcM$. Also, let $u\in l(V_0)$.
	
	First, we show that
	\begin{equation}\label{eqn32}
	\limsup_{n\to\infty}\mcR(r_n,\mcD_n)(u)\leq \mcR(r,\mcD)(u).
	\end{equation}
	We define $f$ to be the unique function in $Dom(\Psi_r\mcD)$ such that $f|_{V_0}=u$ and
	\[\mcR(r,\mcD)(u)=\Psi_r\mcD(f).\]
	By Proposition \ref{prop33}, for any $\varepsilon>0$, there is $f_\varepsilon$ such that $f_\varepsilon|_{V_0}=u$,  $f_\varepsilon\in Dom(\Psi_{r_n}\mcD_n)$ for any $n\geq 1$,  and 
    \[\Psi_r\mcD(f_\varepsilon)\leq\Psi_r\mcD(f)+\varepsilon.\]
    As a consequence, we have 
    \[\limsup_{n\to\infty}\mcR(r_n,\mcD_n)(u)\leq \lim_{n\to\infty}\Psi_{r_n}\mcD_n(f_\varepsilon)=\Psi_r\mcD(f_\varepsilon)\leq \mcR(r,\mcD)(u)+\varepsilon,\]
    where the equality is due to the dominated convergence theorem. Since $\varepsilon$ can be arbitrarily chosen, we get (\ref{eqn32}). 
    
    Next, for each $n$, let $f_n$ be the unique function in $Dom(\Psi_{r_n}\mcD_n)$ such that $f_n|_{V_0}=u$ and 
    \[\mcR(r_n,\mcD_n)(u)=\Psi_{r_n}\mcD_n(f_n).\]
    Then $\{f_n\}_{n\geq 1}$ is uniformly bounded by the Markov property (RF5). In addition, $\Psi_{r_n}\mcD_n(f_n)\leq \mcR(r_*,\mcD_n)(u)$ with $r_*=\inf_{n\geq 1} r_n$, so $\{\Psi_{r_n}\mcD_n(f_n)\}_{n\geq 1}$ is a bounded sequence. By the estimates (\ref{eqn22}) and (\ref{eqn24}), we have
    \[|f_n(p)-f_n(q)|\leq c\Big(d(p,q)^{\frac{\log r^*}{\log \rho}}\sup_{n\geq 1}\Psi_{r_n}\mcD_n(f_n)\Big)^{1/2},\quad\forall n\geq 1,\forall p,q\in \bar{V}_1,\]
    where $r^*=\sup_{n\geq 1}r_n$ and $c^2=\sup_{n\geq 1}\{\frac{4}{{r_n^2}(1-r_n)}R_0^*(\mcD_n)\}$, and so $\{f_n\}_{n\geq 1}$ is also equicontinuous. Thus, there is a subsequence $\{f_{n_k}\}_{k\geq 1}$ such that $f_{n_k}$ converges uniformly to a function $f\in C(\bar{V}_1)$. Clearly, $f$ is an extension of $u$. By Fatou's lemma,  
    \[\mcR(r,\mcD)(u)\leq\Psi_r\mcD(f)\leq \liminf_{k\to\infty}\Psi_{r_{n_k}}\mcD_{n_k}(f_{n_k})=\liminf_{k\to\infty}\mcR(r_{n_k},\mcD_{n_k})(u).\]
    Combining this with (\ref{eqn32}), we  see that 
    \[\mcR(r,\mcD)(u)=\lim_{k\to\infty}\mcR(r_{n_k},\mcD_{n_k})(u).\]
    Since the argument works for any sequence $(r'_n,\mcD'_n)\to (r,\mcD)$, we actually have 
    \[\mcR(r,\mcD)(u)=\lim_{n\to\infty}\mcR(r_n,\mcD_n)(u).\]
    The theorem follows immediately since $u$ can be any function in $l(V_0)$.  
\end{proof}

\section{A fixed point problem}
In this section, analogous to the case of p.c.f. self-similar sets (see\cite{ki3,Sabot}), we consider the renormalization equation
\begin{equation}\label{eqn41}
\mcR_r\mcD=\lambda\mcD,
\end{equation}
with $\lambda>0$. We will prove that for each given $\rho<r<1$, there always exists a positive $\lambda$ such that (\ref{eqn41}) has a solution $\mcD$ in $\mathcal M$. Nevertheless, this is not enough for the construction of a satisfying resistance form on $\G$ for later purpose. In order that cells of same size will be assigned with same renormalization factors, we will in addition require $\lambda=r^2$, i.e.
\begin{equation}\label{eqn42}
\mcR_r\mcD=r^2\mcD.
\end{equation}
The existence and uniqueness of such solution is the main purpose of this section. 

It is natural to look at the symmetric resistance forms on $V_0$, which means that $a_{0,1}=a_{0,2}$ in (\ref{eqn21}). We denote $\mcM_S$ for the set of all such resistance forms. 

\begin{theorem}\label{thm40}
(a). For each $\rho<r<1$, there exists a unique pair of $\lambda(r)$ and $\mcD(r)\in\mcM$ (up to constants) satisfying  (\ref{eqn41}), where $\lambda(r)$ is decreasing and continuous in $r$, and $\mcD(r)$ is in $\mcM_S$. \\ (b). There exists a unique $\rho<r<1$ such that (\ref{eqn42}) has a unique (up to constants) solution $\mcD\in \mcM$. 
\end{theorem}

We will first prove that for each $r$, there exist a unique $\lambda(r)$ such that (\ref{eqn41}) has a solution $\mcD(r)$ in $\mcM_S$, then prove that 
$\mcD(r)$ is indeed the unique solution (up to constants) in $\mcM$. The existence and uniqueness of a solution to (\ref{eqn42}) will follow from the properties of $\lambda(r)$. We divide these into two subsections. 

\subsection{The existence of a symmetric solution}

We begin with some simple observations.

\begin{lemma}\label{lemma41}
	Let $\rho<r<1$ be fixed, and suppose that there is a solution to (\ref{eqn41}). Then the constant $\lambda$ depends only on $r$.
\end{lemma}
\begin{proof}
This follows from a standard argument like the finite graph case \cite{M1}. Suppose that $\mcD,\mcD'$ are two solutions to (\ref{eqn41}) with $\lambda,\lambda'$ being the corresponding constant. Let $u\in l(V_0)$ so that $\frac{\mcD'(u)}{\mcD(u)}=\sup_{v\neq constants}\frac{\mcD'(v)}{\mcD(v)}:=\theta$, and let $f$ be the harmonic extension of $u$ with respect to $\Psi_r\mcD$. Then
\[\lambda'\mcD'(u)=\mcR_r\mcD'(u)\leq \Psi_r\mcD'(f)\leq \theta\Psi_r\mcD(f)=\theta\mcR_r\mcD(u)=\theta \lambda\mcD(u).\]
This implies that $\lambda'\leq \lambda$. A same argument also shows that $\lambda\leq \lambda'$.
\end{proof}

Inspired by Lemma \ref{lemma41}, we can view the constant $\lambda$ in (\ref{eqn41}) as a function of $r$. On the other hand, the problem of solvability of (\ref{eqn41}) can be transferred to  a fixed point problem.

\begin{definition}
(a). Define 
\[\tilde{\mcM}_S=\big\{\mcD\in \mcM:\mcD(f)=a\big(f(q_0)-f(q_1)\big)^2+a\big(f(q_0)-f(q_2)\big)^2+(1-a)\big(f(q_1)-f(q_2)\big)^2, 0<a\leq 1\big\},\]
and for $0<s\leq 1$,
\[\tilde{\mcM}_S^{[s,1]}=\big\{\mcD\in \mcM:\mcD(f)=a\big(f(q_0)-f(q_1)\big)^2+a\big(f(q_0)-f(q_2)\big)^2+(1-a)\big(f(q_1)-f(q_2)\big)^2, s\leq a\leq 1\big\}.\]

(b). For each $\mcD\in \mcM_S$, there is a unique constant $c$ such that $c\mcD\in \tilde{\mcM}_S$, and we denote the resulted form $T\mcD$. We define $\tilde{\mcR}_r:\mcM_S\to \tilde{\mcM}_S$ as $\tilde{\mcR}_r=T\circ\mcR_r$. As before, we write $\tilde{R}(r,\mcD)=\tilde{R}_r(\mcD)$.
\end{definition}

The following lemma will play an essential role.
\begin{lemma}\label{lemma43}
For $\rho<r_0<r_1<1$, there exists $0<s\leq 1$ such that $\tilde\mcR:[r_0,r_1]\times \mcM_S\to\tilde{\mcM}_S^{[s,1]}$.
\end{lemma}
\begin{proof}
Let $\mcD\in \mcM_S$, $r_0\leq r\leq r_1$ and $R$ be the resistance metric on $V_1$ associated to $\Psi_r\mcD$. For convenience, we write $\mcD(f)=a\big(f(q_0)-f(q_1)\big)^2+a\big(f(q_0)-f(q_2)\big)^2+b\big(f(q_1)-f(q_2)\big)^2$, with $a>0,b\geq 0$.   

First, an immediate observation shows that for any $f\in l(V_1)$,
\[\Psi_{r}\mcD(f)\geq \sum_{n=0}^\infty a{r}^{-n}\big(f(F_1^nq_0)-f(F_1^{n+1}q_0)\big)^2\geq a(1-r)\big(f(q_0)-f(q_1)\big)^2,\]
so we have $R(q_0,q_1)\leq\frac{1}{a(1-r)}\leq \frac{1}{a(1-r_1)}$.

Next, let $f$ be the linear function on $\mathbb{R}^2$ such that $f(q_1)=0,f(q_2)=1$ and $f(q_0)=\frac{1}{2}$, so $f$ only depends on the $x$-coordinate.  We introduce an equivalence relation `$\sim_h$' on $V_1$,
\[p\sim_h q\text{ if there exists }w\in W_1\text{ so that }p,q\in \{F_wq_1,F_wq_2\}.\]
Then we modify $f$ on $V_1$ into a function $g\in l(V_1)$ as
\[g(p)=\frac{\sum_{q\sim_h p}f(q)}{\sum_{q\sim_h p}1},\quad\forall p\in V_1.\]
By doing this we have 

\noindent1. $g(p)=g(q)$ if $p\sim_h q$;

\noindent2. $|g(p)-g(q)|\leq c_1\rho^{n}$ if $p,q\in F_wV_0$ with $w\in W_{1,n}$.  

Thus we have 
\[
\begin{aligned}
\Psi_r\mcD(g)&=\sum_{l=1}^\infty\sum_{w\in W_{1,l}}r^{-l+1}\mcD(g\circ F_w)\\
&\leq 2c_1^2 a\sum_{l=1}^\infty r^{-l+1}\rho^{2l}\#W_{1,l}\\
&\leq c_2a\sum_{l=1}^\infty r^{-l}\rho^{l}=\frac{c_2\rho}{r-\rho}a\leq \frac{c_2\rho}{r_0-\rho}a,
\end{aligned}
\]
where we use the estimate $\#W_{1,l}\asymp \rho^{-l}$. Thus, $g$ extends to $g\in C(\bar{V}_1)$ by Proposition \ref{prop23}, and it is direct to check that $g|_{V_0}=f|_{V_0}$. As a consequence, we get $R(q_1,q_2)\geq\frac{r_0-\rho}{c_2\rho}a^{-1}$. 

Due to the above two estimates, there exists $c_3>0$ independent of $\mcD$ such that 
\[\frac{R(q_0,q_1)}{R(q_1,q_2)}\leq c_3.\]
The lemma follows immediately.
\end{proof}

By using Lemma \ref{lemma41}, Lemma \ref{lemma43} and Theorem \ref{thm32}, we can easily prove the following proposition.
\begin{proposition}\label{prop44}
	Let $\rho<r<1$, there always exists a solution to (\ref{eqn41}) in $\mcM_S$, with $\lambda$ uniquely determined by $r$. In addition, regarding $\lambda$ as a function of $r$,  $\lambda(r)$ is decreasing and continuous in $r$. 
\end{proposition}
\begin{proof}
First, we have $\tilde{\mcR}_r:\tilde{\mcM}_S^{[s,1]}\to \tilde{\mcM}_S^{[s,1]}$ for some $s>0$ by Lemma  \ref{lemma43}. The existence of a fixed point of $\tilde{\mcR}_r$ is then an immediate consequence. 

Next, let $r_1<r_2$, and assume that $\mcR_{r_1}\mcD_1=\lambda(r_1)\mcD_1$ and $\mcR_{r_2}\mcD_2=\lambda(r_2)\mcD_2$.  Let $u\in l(V_0)$ so that $\frac{\mcD_2(u)}{\mcD_1(u)}=\sup_{v\neq constants}\frac{\mcD_2(v)}{\mcD_1(v)}:=\theta$, and let $f$ be the harmonic extension of $u$ with respect to $\Psi_{r_1}\mcD_1$, then we have 
$\lambda(r_2)\mcD_2(u)\leq\Psi_{r_2}\mcD_2(f)\leq \theta\Psi_{r_1}\mcD_1(f)=\theta \lambda(r_1)\mcD_1(u)$. So we get $\lambda(r_2){\leq} \lambda(r_1)$. 

Finally, let $r_n\to r$, and let $\mcD_n\in \tilde{\mcM}_S$ be a sequence of solutions to $\mcR_{r_n}\mcD_n=\lambda(r_n)\mcD_n$. Clearly, we have $\rho<\inf_{n\geq 1}r_n<\sup_{n\geq 1}r_n<1$, so $\{\mcD_n\}_{n\geq 1}\subset \tilde{\mcM}_S^{[s,1]}$ for some $s>0$ by Lemma \ref{lemma43}. Thus, there exists a subsequence $\{n_k\}_{k\geq 1}$ such that $\mcD_{n_k}$ converges to some $\mcD\in \tilde{\mcM}_S$ and $\lambda(r_{n_k})$ converges. By Theorem \ref{thm32}, we conclude that $\mcR_r\mcD=\big(\lim_{k\to\infty}\lambda({r_{n_k}})\big)\mcD$. So $\lambda(r)=\lim_{k\to\infty}\lambda(r_{n_k})$. Since the argument works for any sequence $r_n\rightarrow r$, $\lambda(r)$ is continuous in $r$.
\end{proof}

We have an easy estimate of $\lambda(r)$.

\begin{lemma}\label{lemma45}
	For $\rho<r<1$, we have $\big(\frac{1}{1-r}-\frac{r}{2+2r+2r^2}\big)^{-1}\leq \lambda(r)\leq \frac{2}{2+r}$.
\end{lemma}
\begin{proof}
	We consider a function $u\in l(V_0)$ with $u(q_0)=0$ and $u(q_1)=u(q_2)=1$. Without loss of generality,   we assume the solution $\mcD$ has  $\mcD(u)=2$. 
	
	To get the upper bound of $\lambda(r)$, we construct an extension $f\in l(V_1)$ of $u$ as 
	\[f(p)=\begin{cases}
	0,&\text{ if }p=q_0,\\
    \frac{2}{2+r},&\text{ if }p\in \{F_1q_0,F_2q_0\},\\
	1,&\text{ if }p\in F_1 V_1\cup F_2 V_1.
	\end{cases}\]
	Then the upper bound follows easily from the following estimate,
	\[\mathcal{R}_\lambda\mcD(u)\leq \Psi_r\mcD(f)=\frac{4}{2+r}=\frac{2}{2+r}\mcD(u).\] 
	
	To get the lower bound, we look at a subgraph in $(V_1,\sim)$, whose vertices are $\{F_i^lq_0\}_{i,l}$ with $i\in\{1,2\}$ and $l\geq 0$, together with
	\[\begin{aligned}
	p_{i,0}=F_iq_0,\quad p_{i,1}=F_iF_jq_0,\quad p_{i,2}=F_iF_jF_iq_0,\\ p_{i,3}=F_iF_jF_i^2q_0, \quad p_{i,4}=F_i^2F_jq_0,\quad p_{i,5}=F_i^2q_0,
	\end{aligned}\]
	with $i,j\in \{1,2\}$ and $j\neq i$, and edges inherit from $(V_1,\sim)$. Let $f\in l(V_1)$ be the harmonic extension of $u$, denote $c_l=r^{-l-1}$ for $l\in\{0,1,2\}$, and $c_l=r^{l-6}$ for $l\in\{3,4\}$, then the lower bound follows from the estimate that
	\[\begin{aligned}
	\mcR_r\mcD(u)=\Psi_r\mcD(f)\geq &\sum_{i=1,2}\Big(\sum_{l=0}^\infty r^{-l}\big(f(F_i^lq_0)-f(F_i^{l+1}q_0)\big)^2+\sum_{l=0}^4 c_l\big(f(p_{i,l})-f(p_{i,l+1})\big)^2\Big)\\
	&\geq 2\big(\frac{1}{1-r}-\frac{r}{2+2r+2r^2}\big)^{-1}=\big(\frac{1}{1-r}-\frac{r}{2+2r+2r^2}\big)^{-1}\mcD(u),
	\end{aligned}\]
	where the last inequality can be done by an easy computation of the effective resistances on the subgraph.
\end{proof}

Using Proposition \ref{prop44} and Lemma \ref{lemma45}, we arrive the main result of this subsection, concerning the solvability of (\ref{eqn42}). 
\begin{theorem}\label{thm46}
There exists a unique $\rho<r<1$ such that (\ref{eqn42}) has a solution $\mcD\in\mcM_S$. 
\end{theorem}
\begin{proof}
By Proposition \ref{prop44}, we see that there is a continuous function $\lambda(r)$ so that $\mcR_\lambda(\mcD)=\lambda(r)\mcD$ has a solution. Noticing that 
\[\big(\frac{1}{1-\rho}-\frac{\rho}{2+2\rho+2\rho^2}\big)^{-1}>1-\rho=\rho^2,\text{ and } \frac{2}{3}<1,\]
there exists $\rho<r<1$ such that $\lambda(r)=r^2$ by Lemma \ref{lemma45}. The uniqueness follows from the fact that $\lambda(r)$ is decreasing in $r$, while $r^2$ is strictly increasing.  
\end{proof}

\noindent\textbf{Remark.} We can see the uniqueness of $r$ from another point of view. Let $\theta=\frac{\log r}{\log \rho}$, we will see in Section 7 that $\theta+d_H$ is the \textit{walk dimension} of the resulting diffusion process on the metric measure space $(\G,d,\mu_H)$, whose uniqueness is shown in \cite{GHL} under some weak conditions.

\subsection{The uniqueness}
In this subsection, we consider the uniqueness of the solution to (\ref{eqn41}) or $(\ref{eqn42})$. The proof is inspired by Sabot's work \cite{Sabot}.
\begin{theorem}\label{thm47}
Let $\rho<r<1$ and $\mcD\in\mcM_S$ be a symmetric solution to (\ref{eqn41}). Then $\mcD$ is the unique solution in $\mcM$ to (\ref{eqn41}).
\end{theorem}

For fixed $\rho<r<1$ and $\mcD\in\mcM_S$ satisfying (\ref{eqn41}), for convenience, we always write 

1. $h_s$ for the harmonic function with $h_s(q_0)=0$, $h_s(q_1)=h_s(q_2)=1$, and denote $E_s=\{f\in l(V_0): f(q_0)=0,f(q_1)=f(q_2)=c, c\in \mathbb{R}\}$;

2. $h_a$ for the harmonic function with $h_a(q_0)=0$, $h_a(q_1)=-h_a(q_2)=1$, and denote $E_a=\{f\in l(V_0): f(q_0)=0,f(q_1)=-f(q_2)=c, c\in \mathbb{R}\}$.

Both $h_s,h_a$ are harmonic with respect to $\Psi_r\mcD$ on $\bar{V}_1\setminus V_0$, i.e. $\Psi_r\mcD(h_s,f)=\Psi_r\mcD(h_a,f)=0$
for any $f\in Dom(\Psi_r\mcD)$ such that $f|_{V_0}=0$. 

\begin{lemma}\label{lemma48}
	For $r$, $\mcD$ as above, we have 
	\[h_s(F_1q_0)=h_s(F_2q_0)=\lambda(r),\quad h_a(F_1q_0)=-h_a(F_2q_0)=\eta,\]
	for some $|\eta|<\lambda(r)$.
\end{lemma}
\begin{proof}
For convenience, we write $\mcD$ in the form $\mcD(f)=a\big(f(q_0)-f(q_1)\big)^2+a\big(f(q_0)-f(q_2)\big)^2+b\big(f(q_1)-f(q_2)\big)^2$, with $a>0,b\geq 0$. 

First, let $h=1-h_s$, we have $\mcR_r\mcD(h_s,h)=-2a\lambda(r)$. On the other hand, let $f\in l(V_1)$ be defined as $f(p)=\delta_{q_0,p}$, then clearly $f\in Dom(\Psi_r\mcD)$, and $f|_{V_0}=h|_{V_0}$. Since $h_s$ is harmonic, 
\[\Psi_r\mcD(h_s,h)=\Psi_r\mcD(h_s,f)=-ah_s(F_1q_0)-ah_s(F_2q_0).\]
This shows the first assertion since $\mcR_r\mcD(h_s,h)=\Psi_r\mcD(h_s,h)$. 

Next, by the symmetry of $\mcD$, there exists a number $\eta$ such that $h_a(F_1q_0)=-h_a(F_2q_0)=\eta$. We need to show that $|\eta|<\lambda(r)$. We consider the matrix $M$ such that 
\[\big(h(F_1q_0),h(F_2q_0)\big)^t=M\big(h(q_1),h(q_2)\big)^t,\]
holds for any harmonic function $h$ with $h(q_0)=0$. Due to the Perron-Frobenius theorem, it suffices to show that each entry of $M$ is positive. This can be deduced by proving the harmonic function $h_1$ with boundary value $h_1(q_1)=1,h_1(q_0)=h_1(q_2)=0$ is positive on $V_1\setminus V_0$. To see this, we assume there exists $p\in V_1\setminus V_0$ such that $h(p)=0$. Let $\psi_p\in Dom(\Psi_r\mcD)$ be defined as $\psi_p(q)=\delta_{p,q}$, then $\Psi_r\mcD(\psi_p,h_1)=0$, so $h_1(p)$ is the weighted average of its neighbours. Thus $h_1$ is zero on the neighbours of $p$. Repeating the argument, we see that $h|_{V_1}=0$. A contradiction. 
\end{proof}

\begin{proof}[Proof of Theorem \ref{thm47}]
Assume there is another solution $\mcD'\in\mcM$ to (\ref{eqn41}). 

Firstly, we will show that $\mcD'$ is also symmetric. By diagonalizing $\mcD'$ with respect to $\mcD$, we have two $1$-dimensional subspaces $L_1,L_2$ of $l(V_0)$ such that   

\noindent1. $L_1,L_2$ are orthogonal with respect to both $\mcD$ and $\mcD'$;

\noindent2. $\mcD'|_{L_1}=\kappa_1 \mcD|_{L_1}$ and $\mcD'|_{L_2}=\kappa_2\mcD|_{L_2}$, with $0<\kappa_1<\kappa_2$. 

Let $u\in L_2$ and $h_u$ be the harmonic extension of $u$ with respect to $\Psi_r\mcD$. Then
\[\begin{aligned}
\lambda(r)\mcD'(u)&=\kappa_2 \lambda(r)\mcD(u)=\kappa_2\Psi_r\mcD(h_u)=\sum_{w\in W_{1}} r^{-|w|+1}\kappa_2\mcD(h_u\circ F_w)\\
&\geq \sum_{w\in W_{1}} r^{-|w|+1}\mcD'(h_u\circ F_w)=\Psi_r\mcD'(h_u)\geq \lambda(r)\mcD'(u).
\end{aligned}\]
Clearly, this implies that $h_u\circ F_w\in L_2+constants$ for each $w\in W_1$. In particular, we have $h_u\circ F_0\in L_2+constants$, which means $L_2+constants$ is an invariant space under the mapping $u$ to $h_u\circ F_0$. By Lemma \ref{lemma48}, we see that $L_2+constants$ is either $E_s+constants$ or $E_a+constants$. Thus, we have $\mcD'\in \mcM_S$.

Secondly, from the above argument, it is not hard to see that $h_s\circ F_w\in E_s+constants$ and $h_a\circ F_w\in E_a+constants$, for any $w\in W_1$.

Lastly,  arbitrarily pick a $\tilde{\mcD}\in \mcM_S$, we will prove that $\tilde{\mcD}$ must also solve (\ref{eqn41}), which obviously contradicts Lemma \ref{lemma43}. To achieve this purpose, let $\tilde{h}_s$ and $\tilde{h}_a$ be the harmonic functions with respect to $\Psi_r\tilde{\mcD}$,  with the same boundary value on $V_0$ as $h_s,h_a$. By following a same argument as Lemma 5.9 in Sabot's paper \cite{Sabot}, we can see that $\tilde{h}_s=h_s$ and $\tilde{h}_a=h_a$. For convenience of readers, we reproduce the proof here. Write $g=h_s-\tilde{h}_s$. Also, for each $w\in W_1$, let $g_{w,s}\in E_s+constants, g_{w,a}\in E_a+constants$ such that $g_w=:g\circ F_w=g_{w,s}+g_{w,a}$. Then, we can see that 
\[\begin{aligned}
\Psi_r\tilde{\mcD}(g)&=\Psi_r\tilde{\mcD}(h_s,g)=\sum_{w\in W_{1}}r^{-|w|+1}\tilde{\mcD}(h_{s,w},g_w)=\sum_{w\in W_{1}}r^{-|w|+1}\tilde{\mcD}(h_{s,w},g_{w,s})\\&=c\sum_{w\in W_{1}}r^{-|w|+1}\mcD(h_{s,w},g_{w,s})=c\sum_{w\in W_{1}}r^{-|w|+1}\mcD(h_{s,w},g_w)=c\Psi_r\mcD(h_s,g)=0,
\end{aligned}\] 
for some constant $c$, with $h_{s,w}=: h_s\circ F_w$.
Thus $g=0$ as desired. As a consequence, we can easily see that, $\tilde{\mcD}$ is a solution to (\ref{eqn41}), so we arrive the desired contradiction. \end{proof}

Finally, Theorem \ref{thm40} immediately follows from Proposition \ref{prop44}, Theorem \ref{thm46} and \ref{thm47}.

\section{Construction of the Dirichlet form on $\G$}
We will construct a resistance form on the golden ratio Sierpinski gasket $\G$ in this section. Let $\rho<r<1$, $\mcD$ be the unique solution to (\ref{eqn42}), i.e. $\mcR_r\mcD=r^2 \mcD$. We will focus on this standard form in most contents. For short, we write 
\[\theta=\frac{\log r}{\log \rho},\quad \rho_w=\prod_{n=1}^{|w|}\rho_{w_n},\quad r_w=\rho_w^{\theta},\]
with $\rho_{0}=\rho^2$ and $\rho_1=\rho_2=\rho$. Obviously, $\rho_w$ is the contraction ratio of $F_w$.

The following definition is similar to the construction in \cite{ki3}, though we use the infinite graphs at each level.

\begin{definition}\label{def51}
 (a). For $m\geq 0$ and $f\in C(\bar{V}_m)$, we write $\mcD^{(m)}(f)=\sum_{w\in W_m}r_w^{-1}\mcD(f\circ F_w)$, and $\mcF^{(m)}=\{f\in C(\bar{V}_m):\mcD^{(m)}(f)<\infty\}$. In addition, for $f,g\in \mcF^{(m)}$, we define
\[\mcD^{(m)}(f,g)=\sum_{w\in W_m}r_w^{-1}\mcD(f\circ F_w,g\circ F_w).\]
	
(b). Define $\mathcal{F}=\{f\in C(\G):\lim_{m\to\infty}\mcD^{(m)}(f)<\infty\}$. For $f,g\in \mathcal{F}$, define 
\[\mcE(f,g)=\lim_{m\to\infty}\mcD^{(m)}(f,g).\]
\end{definition}

The limit in (b) always exists due to fact that 
\[\mcD^{(m+1)}(f)=\sum_{w\in W_{m+1}}r_w^{-1}\mcD(f\circ F_w)=\sum_{w\in W_m}r_w^{-1}r^{-2}\Psi_r\mcD(f\circ F_w)\geq \sum_{w\in W_m}r_w^{-1}\mcD(f\circ F_w)=\mcD^{(m)}(f).\]
In the rest of this section, we will show that $(\mcE,\mathcal{F})$ is a good form.

\begin{lemma}\label{lemma52}
	For $m\geq 0$, $(\mcD^{(m)},\mcF^{(m)})$ is a resistance form on $\bar V_m$. In addition, let 
	\[R_m(p,q)=\sup_{f\in \mcF^{(m)}}\frac{|f(p)-f(q)|^2}{\mcD^{(m)}(f)},\]
    then we have $R_n(p,q)=R_m(p,q)$ if $p,q\in \bar{V}_m$ and $n\geq m$.
\end{lemma}
\begin{proof}
	(RF1) and (RF5) are trivial. We only need to verify (RF2)-(RF4). For convenience, we focus on $(\mcD^{(2)},\mcF^{(2)})$ only, while for larger $m$, a same proof works inductively. 
	
	(RF2). Let $\{f_k\}_{k\geq 1}$ be a Cauchy sequence in $\mcF^{(2)}$. Then,  $f_k|_{\bar{V}_1}$ converges in $\mcF^{(1)}$ to some $\tilde{f}$ in $\mcF^{(1)}$, since $(\mcD^{(1)},\mcF^{(1)})$ is a resistance form. Also, for each $w\in W_1$, $f_k\circ F_w$ converges in $\mcF^{(1)}$ to a function $\tilde{f}_w$. Now, define $f\in l(\bar V_2)$  such that $f\circ F_w=\tilde{f}_w$ and $f|_{\bar{V_1}\setminus V_1}=\tilde{f}$.  
	
	We show that $f\in C(\bar V_2)$. It suffices to prove that $f$ is continuous at any point $p\in \bar{V}_1\setminus V_1$. In fact, for any $\varepsilon$, there exists $\delta$ and $N$ such that 
	
	1. for $q\in B_{\delta}(p)\cap\bar{V}_1$, we have $|f(p)-f(q)|<\varepsilon$;
	
	2. for $w\in \bigcup_{n=N}^\infty W_{1,n}$ and $q,q'\in F_w\bar{V}_1$, we have $|f(q)-f(q')|<\varepsilon$. This follows from the fact that $\mcD^{(1)}(f\circ F_w)\leq r_w\sup_{k\geq 1}\mcD^{(2)}(f_k)$.
	
	The continuity of $f$ follows immediately. Lastly, by using Fatou's lemma, we can directly check that $f_k$ converges to $f$ in $\mcF^{(2)}$. 
	
	(RF3). First, we observe that the minimal energy extension of $f\in \mcF^{(1)}$ to $l(V_2)$ is continuous by a same argument as in (RF2). So we have enough functions in $\mcF^{(2)}$.
	
	Let $V$ be a finite set and $u\in l(V)$. First, we always have $f_1\in \mcF^{(1)}$ such that $f_1|_{V\cap \bar{V}_1}=u|_{V\cap \bar{V}_1}$. Then we can extend $f_1$ to be a desired function in $\mcF^{(2)}$.
	
	(RF4). Let $p,q\in \bar{V}_2$ and $f\in\mcF^{(2)}$. If $p\in \bar{V}_1$, we let $p'=p$; otherwise we choose $p'\in V_1$ so that $p,p'\in F_w\bar{V_1}$ for some $w\in W_1$, and thus
	\[\mcD^{(2)}(f)\geq r_w^{-1}\mcD^{(1)}(f\circ F_w)\geq c_1\big(f(p)-f(p')\big)^2,\]
	for some $c_1>0$. Also, we define $q'$ in a same manner. It then follows that 
	\[\mcD^{(2)}(f)\geq c_2\Big(\big(f(p)-f(p')\big)^2+\big(f(p')-f(q')\big)^2+\big(f(q')-f(q)\big)^2\Big)\geq c_3\big(f(p)-f(q)\big)^2.\]
	(RF4) follows immediately. 
	
	Thus, we have proved that $(\mcD^{(2)},\mcF^{(2)})$ is a resistance form on $\bar V_2$. The claim that $R_2(p,q)=R_1(p,q)$ for $p,q\in \bar V_1$ is obvious. The same arguments can be used inductively for $m\geq 3$.
\end{proof}

In some situations, it is convenient to involve words in $\tilde{W}_*$.

\begin{lemma}\label{lemma53}
	Let $w\in \tilde{W}_*$ and $m$ be the number of $0'$s in $w$. Then we have \[\mcD^{(1)}(f\circ F_w)\leq r_w\mcD^{(m+1)}(f),\]
	for any $f\in \mcF^{(m+1)}$. As a consequence, there is a constant $c>0$ such that, for any $p,q\in F_w\bar{V}_1$,  we have
	\[R_{m+1}(p,q)\leq c d(p,q)^\theta.\]
\end{lemma}
\begin{proof}
	Noticing that $\{w\tau:\tau\in W_1\}\subset {W}_{m+1}$, the lemma is obvious by the definition of $\mcD^{(m+1)}$ and Proposition \ref{prop23}.
\end{proof}

Using Lemma \ref{lemma52} and \ref{lemma53}, we have the following estimate of the resistance metric. 

\begin{lemma}\label{lemma54}
	For $m\geq 0$ and $p,q\in \bar{V}_m$, define $\tilde{R}(p,q)=R_m(p,q)$. Then $\tilde{R}(p,q)$ is well defined on $(\bigcup_{m\geq 0}\bar{V}_m)\times (\bigcup_{m\geq 0}\bar{V}_m)$, and we have $\tilde{R}(p,q)\leq cd(p,q)^\theta$ for some $c>0$.
\end{lemma}  
\begin{proof}
	First, we claim that there is a constant $c_1>0$ such that 
	\[\tilde{R}(p,q)\leq c_1\rho_w^\theta,\quad  \forall w\in \tilde{W}_*, \forall p,q\in F_w \G\cap (\bigcup_{m\geq 0}\bar{V}_m).\] 
    We first consider the case $q\in F_w\bar{V}_1$. Assume that $p\in F_w\bar{V}_n$ for some $n\geq 1$, then we can find $\tau\in W_{n-1}$ such that $p\in F_wF_\tau \bar{V}_1$. We can then find a sequence 
	\[q=p_0,p_1,\cdots,p_{|\tau|+1}=p,\]
	such that $p_i\in F_wF_{[\tau]_{i-1}}\bar{V}_1\cap F_wF_{[\tau]_{i}}\bar{V}_1$ for $1\leq i\leq |\tau|$.  As a consequence, by using Lemma \ref{lemma53}, we see that 
	\[\tilde{R}(p,q)\leq \sum_{i=0}^{|\tau|} c_2d(p_i,p_{i+1})^\theta\leq  \sum_{i=0}^{|\tau|} c_2(\rho_w\rho^i)^\theta\leq \frac {c_2}{1-\rho^\theta}\rho_w^\theta,\]
	where $c_2$ is the same constant in Lemma \ref{lemma53}.
	For general $q$, we only need to set $c_1=\frac{2c_2}{1-\rho^\theta}\rho_w^\theta$. 
	
	Now, let $p,q\in \bigcup_{m=0}^\infty \bar{V}_m$. We choose $w,w'\in\tilde W_*$ such that $p\in F_w\G$, $q\in F_{w'}\G$ and 
	\[\rho d(p,q)\leq\rho_w,\rho_{w'}<\rho^{-1}d(p,q).\]
	In addition, we can find a chain 
	\[w=w^{(0)},w^{(1)},\cdots,w^{(k)}=w'\]
	such that $\min\{\rho_w,\rho_{w'}\}\leq\rho_{w^{(i)}}<\rho^{-2}\min\{\rho_w,\rho_{w'}\}$ of length at most $c_3$, where $c_3$ is a constant independent of $p,q$. By choosing a sequence $p=p_0,p_1,\cdots,p_{k+1}=q$ such that $p_i\in F_{w^{(i-1)}}\bar{V}_1\cap F_{w^{(i)}}\bar{V}_1$, $1\leq i\leq k$, we get the desired estimate as above. 
\end{proof}

Now, we can show that $(\mcE,\mcF)$ is a good form.
\begin{theorem}\label{thm55}
	 $(\mcE,\mcF)$ defined in Definition \ref{def51} is a strongly local regular resistance form on $\G$. 
\end{theorem}
\begin{proof}
	First, we claim that $(\mcE,\mcF)$ is a resistance form on $\bigcup_{m\geq 0}\bar{V}_m$. (RF1) and (RF5) are obvious. Observing that by keeping doing the minimal energy extension, we can extend any $f\in \mcF^{(m)}$ to $f\in \mcF$ thanks to the upper bound estimate of the resistance metric in Lemma \ref{lemma54}. (RF2), (RF3) and (RF4) are then easy to shown with Lemma \ref{lemma52}. In addition, we see that 
	\[\tilde{R}(p,q)=R(p,q):=\sup_{f\in \mcF}\frac{|f(p)-f(q)|^2}{\mcE(f)},\quad \forall p,q\in \bigcup_{m\geq 0}\bar{V}_m.\]
	
	Next, to prove that $(\mcE,\mcF)$ is a resistance form on $\G$, we need to show that $\mcF$ separates points in $\G$, just like in Proposition \ref{prop23}. It suffices to prove that $\mcF$ is dense in $C(\G)$. Let $u\in C(\G)$, we fix $N$ large enough so that $|u(x)-u(y)|<\varepsilon$ if $x,y
	\in F_wK$ and $|w|\geq N$. We can apply Proposition \ref{prop33} to create $f\in \mcF$ such that $\|f-u\|_{L^\infty(\G)}< 2\varepsilon$. First we find $f_1\in \mcF^{(1)}$ such that 
	
	1. $\|f_1-u|_{\bar{V_1}}\|_{L^\infty}<\varepsilon$;
	
	2. $f_1(p)=u(p)$ for any $p\in \bigcup_{n=1}^N\bigcup_{w\in W_{1,n}} F_wV_0$. 
	
	\noindent Then we apply harmonic extension to $f_1$ on $\bar{V}_2\setminus \bigcup_{n=1}^N\bigcup_{w\in W_{1,n}} F_w\bar{V}_1$. On the cells $F_w\bar{V}_1$ with $|w|<N$, we apply a same construction to get $f_2$, but with $N-2$ replacing $N$ this time. After  $k=[N/2]+1$ times, we get $f_{k}\in \mcF^{(k)}$ such that $\|f_k-u|_{\bar{V}_k}\|_{L^\infty}<2\varepsilon$. Since all cells have size smaller than $\rho^N$, by doing harmonic extension, we get $f\in \mcF$ such that $\|f-u\|_{L^\infty(\G)}<2\varepsilon$. 
	Thus, $(\mcE,\mcF)$ is  regular resistance form on $\G$. 
	
	It remains to show that the form is strongly local. Let $f,g\in \mcF$ with $supp(f)\cap supp(g)=\emptyset$, then there exists $\varepsilon>0$ such that $d\big(supp(f),supp(g)\big)>\varepsilon$. Thus, we have $\mcD^{(n)}(f,g)=0$ for large $n$. By taking the limit, we see that $\mcE(f,g)=0$. Clearly $1\in \mcF$ with $\mcE(1)=0$, and it follows that the form is strongly local. 
\end{proof}

In the remaining of this section, we would like to characterize $(\mcE,\mcF)$ as the unique self-similar form associated with the infinite i.f.s. $\{F_w\}_{w\in W_1}$.

\begin{theorem}\label{thm56}
The resistance form $(\mcE,\mcF)$ satisfies the following properties:

(a). $\mcF\subset C(\G)$.

(b). For each $f\in \mcF$, we have $f\circ F_w\in\mcF$ for all $w\in W_1$, and in addition,  
\[\mcE(f)=\sum_{w\in W_1}\rho_w^{-\theta}\mcE(f\circ F_w).\]

(c). Reversely, let $f\in C(\G)$, if $f\circ F_w\in \mcF$ for all $w\in W_{1}$, and $\sum_{w\in W_1}\rho_w^{-
\theta}\mcE(f\circ F_w)<\infty$, then $f\in \mcF$. 

\noindent Moreover, $(\mcE,\mcF)$ (up to constants) and $\theta$ are uniquely determined by the above properties.
\end{theorem}
\begin{proof}
The claimed properties of $(\mcE,\mcF)$ are immediate consequences of the construction. 

The uniqueness follows by a well-known argument, but in the infinite graph version. Let $(\mcE',\mcF')$ be another form satisfying the above properties with $\theta'$ replacing $\theta$. Define $\mcD'$ to be the trace of $\mcE'$ onto $V_0$, and write $r'_w=\rho_w^{\theta'}$, $r'=\rho^{\theta'}$. For any $u\in l(V_0)$, let $h_u$ be the harmonic extension of $u$ to $\mcF'$, then we can see that  
\[\mcD'(u)=\mcE'(h_u)=\sum_{w\in W_1} {r'}_w^{-1}\mcE'(h_u\circ F_w)\geq \sum_{w\in W_1}{r'}_w^{-1}\mcD'\big((h_u\circ F_w)|_{V_0}\big)\geq{r'}^{-2}\mcR_{r'}\mcD'(u),\]
where $\mcR_{r'}$ is the renormalization map introduced in Definition \ref{def31}, and we use properties (a) and (b) in the inequalities.

On the other hand, we can do the harmonic extension of $u$ in two steps: first we extend $u$ to $f_1\in C(\bar{V}_1)$ so that $f_1$ minimizes $\Psi_{r'}\mcD'$, then we take harmonic extension of $f_1$ on each cell $F_w\G,w\in W_1$, to $f\in C(\G)$, by using property (a) and the Markov property (RF5). In addition, $f\in \mcF'$ by the property (c). Then, by property (b),
\[{r'}^{-2}\mcR_{r'}\mcD'(u)={r'}^{-2}\Psi_{r'}\mcD'(f_1)=\sum_{w\in W_1}{r'}_w^{-1}\mcE'(f\circ F_w)=\mcE'(f)\geq \mcD'(u).\]

Thus, we get $\mcR_{r'}\mcD'={r'}^2\mcD'$, which implies that $\mcD'=\mcD$ and $\theta'=\theta$ by Theorem \ref{thm40}. Finally, by a similar argument, one can easily find that the restriction of $\mcE'$ to $\bar{V}_m$ is $\mcD^{(m)}$, and the claim that $\mcE'=\mcE$ follows immediately by taking the limit. 
\end{proof}
 
Finally, the form $(\mcE,\mcF)$ is decimation invariant with respect to the graph-directed construction in Definition \ref{def12}.

\begin{definition}\label{def57}
	Take the same notations as in Definition \ref{def12}. Let $(\mcE_1,\mcF_1)=(\mcE,\mcF)$, and define $(\mcE_2,\mcF_2)$ as follows,
	\[\begin{cases}
	\mcE_2(f,g)=\sum_{w\in W_1,F_w\G\subset K_2} \rho_w^{-\theta}\mcE(f\circ F_w, g\circ F_w),\\
	\mcF_2=\{f\in C(K_2): \text{ }f\circ F_w\in \mcF, \forall w\in W_1 \text{ such that }F_w\G\subset K_2, \text{ }\mcE_2(f)<\infty\}. 
	\end{cases}\]
\end{definition}

It is not hard to verify that $(\mcE_2,\mcF_2)$ is a resistance form on $K_2$. Moreover, we have

\begin{theorem}\label{thm58}
	Take the same notations as in Definition \ref{def12}, and write $\rho_{e_j}$ for the similarity ratio of $\psi_{e_j}$, $1\leq j\leq 6$. Let $(\mcE_i,\mcF_i),i=1,2$ be defined as in Definition \ref{def57}. Then, for $f_i\in \mcF_i$, $i=1,2$, we have $f_{e_{j,1}}\circ \psi_{e_j}\in \mcF_{e_{j,2}}$ for $1\leq j\leq 6$ and
	\[\mcE_1(f_1)=\sum_{j=1}^2\rho^{-\theta}_{e_j}\mcE_{e_{j,2}}(f_1\circ \psi_{e_j}),\quad \mcE_2(f_2)=\sum_{j=3}^6\rho^{-\theta}_{e_j}\mcE_{e_{j,2}}(f_2\circ \psi_{e_j}).\]
	Reversely, let $f_1\in C(K_1)$, if $f_1\circ \psi_{e_j}\in \mcF_{e_{j,2}}$ for $j=1,2$, then $f_1\in \mcF_1$. The same holds for $(\mcE_2,\mcF_2)$.
\end{theorem}

\noindent{\textbf{Remark.} At the end of this section, we remark that a same construction can be applied to get some non-standard self-similar forms on $\G$ with respect to the infinite i.f.s. $\{F_w\}_{w\in W_1}$, by starting with any solution $R_{r'}\mcD'=\lambda(r')\mcD'$. Theorem \ref{thm55} and \ref{thm58} still hold for the forms, with slight changes of the renormalization factors. Nevertheless, the good heat kernel estimate (Theorem \ref{thm64} below) will not hold, but it is possible to get a heat kernel estimate in the form of Hambly and Kumagai's on p.c.f. self-similar sets \cite{HK}. 

\section{Transition density estimate}
Let $\mu_H$ be the normalized Hausdorff measure on $\G$. $(\mcE,\mcF)$ becomes a local regular Dirichlet form on $L^2(\G,\mu_H)$ ($L^2(\G)$ for short) in a standard way. By the celebrated result \cite{FOT}, there is a Hunt process $X=(\mathbb{P}^x,x\in\G,X_t,t\geq 0)$ associated with $(\mcE,\mcF)$ such that 
\[\mathbb{E}^x[f(X_t)]=P_tf(x),\quad \text{ a.e. } x\in\G,\]
where $P_t$ is the associated hear operator. In this last section, we will show that $X$ is a fractional diffusion. We refer to Barlow's book \cite{B}, Section 3, for the definition of this fractional diffusion. 

\begin{definition}\label{def61}
 A Markov process $X=(\mathbb{P}^x,x\in\G,X_t,t\geq 0)$ is a \emph{fractional diffusion} on the fractional metric space $(\G,d_g,\mu_H)$ (see Section 2)  if
	
	(a). $X$ is a conservative Feller diffusion with state space $\G$;
	
	(b). $X$ is $\mu_H$-symmetric;
	
	(c). $X$ has a symmetric transition density $p(t,x,y)=p(t,y,x)$, $t>0$, $x,y\in \G$, which satisfies the Chapman-Kolmogorov equations and is jointly continuous for $t>0$;
	
	(d). there exist a constant $\beta$ and $c_1$-$c_4>0$, such that for $0<t\leq 1$, 
	\[c_1t^{-d_H/\beta}\exp\big(-c_2(\frac{d_g(x,y)^\beta}{t})^{\frac{1}{\beta-1}}\big)\leq p(t,x,y)\leq c_3t^{-d_H/\beta}\exp\big(-c_4(\frac{d_g(x,y)^\beta}{t})^{\frac{1}{\beta-1}}\big),\]
	where $d_H$ is the Hausdorff dimension of $\G$.
\end{definition}

Since $d_g\asymp d$ by Lemma \ref{lemma15}, it suffices to consider the Eculidean metric $d$ in the following.

We will closely follow Barlow's book \cite{B} and Hambly and Kumagai's paper \cite{HK}. We only provide some essential estimates, including a Nash inequality and an estimate of the resistance metric $R$. 

For convenience, for $0<s<1$, we write $\tilde{W}_s=\{w\in \tilde{W}_*:\rho_w\leq s< \rho_{([w]_{|w|-1})}\}$, and by identifying words representing the same cells, we get a quotient class $\hat{W}_s$.  

\begin{proposition}[Nash inequality]\label{prop62}
	Let $d_S=\frac{2d_H}{d_H+\theta}$ with $\theta=\frac{\log r}{\log \rho}$, and $f\in \mcF$, we have 
	\[\|f\|_{L^2(\G)}^{2+4/d_S}\leq c\big(\mcE(f)+\|f\|_{L^2(\G)}^2\big)\|f\|_{L^1(\G)}^{4/d_S},\]
	for some constant $c>0$ independent of $f$.
\end{proposition}

\begin{proof}
	The proof is essentially the same as that for the p.c.f. self-similar sets \cite{HK}. We reproduce it here for convenience of readers. Write $f_w=f\circ F_w$ for $w\in\tilde W_*$ for short. Then for $0<s<1$,
	\[\begin{aligned}
	\|f\|^2_{L^2(\G)}&\leq \sum_{w\in \hat{W}_s}\rho_w^{d_H}\|f_w\|^2_{L^2(\G)}\leq c_1\sum_{w\in \hat{W}_s}\rho_w^{d_H}\big(\mcE(f_w)+\|f_w\|^2_{L^1(\G)}\big)\\
	&\leq c_2s^{d_H+\theta}\sum_{w\in \hat{W}_s}\rho_w^{-\theta}\mcE(f_w)+c_3s^{-d_H}\sum_{w\in \hat{W}_s}(\rho_w^{d_H}\|f_w\|_{L^1(\G)})^2\\
	&\leq c_4\big(s^{d_H+\theta}\mcE(f)+s^{-d_H}\|f\|_{L^1(\G)}^2\big),
	\end{aligned}\]
	where in the last inequality, we use the observation that $\sum_{w\in \hat{W}_s}\rho_w^{-\theta}\mcE(f_w)\leq c'\mcE(f)$ for some $c'\geq 1$. In the case that $\mcE(f)>\|f\|^2_{L^1(\G)}$, we choose $s$ such that $s^{2d_H+\theta}\mcE(f)=\|f\|^2_{L^1(\G)}$, then the result follows immediately. In the case that $\mcE(f)\leq \|f\|^2_{L^1(\G)}$, we have $\|f\|^2_{L^2(\G)}\leq c_1(\mcE(f)+\|f\|_{L^1(\G)}^2)\leq 2c_1\|f\|^2_{L^1(\G)}$, and the result follows. 
\end{proof}

The Nash inequality provides an upper bound estimate $p(t,x,y)\leq c_1t^{-d_S/2}$. In addition, $|p(t,x,y)-p(t,x,y')|\leq c_2t^{-1-d_S/2}R(y,y'), \forall 0<t\leq 1, x,y,y'\in\G$. See \cite{CKS} for a proof.

\begin{proposition}\label{prop63}
	Let $R(\cdot,\cdot)$ be the resistance metric associated with $(\mcE,\mcF)$ on $\G$. Then there exist $c_1, c_2>0$ such that
	\[c_1d(p,q)^\theta\leq R(p,q)\leq c_2d(p,q)^\theta,\quad \forall p,q\in\G.\]
	In addition, for $p\in\G$ and $A\subset \G$, define $R(p,A)=\sup\{\mcE(f)^{-1}:f\in \mcF, f(p)=1,f|_A=0\}$. Then there exists $c_3,c_4>0$ such that
	\[c_3s^\theta\leq R\big(p,B^c_s(p)\big)\leq c_4s^\theta,\]
	where $B_s(p)=\{q\in \G:d(p,q)<s\}$ with $p\in\G$ and $0<s<1$, and $B_s^c(p)$ is the complement of $B_s(p)$ in $\G$.
\end{proposition}
\begin{proof}
	We already have the estimate $R(p,q)\leq c_2d(p,q)^\theta$ in Lemma \ref{lemma54} and Theorem \ref{thm55}. Now we show $R\big(p,B^c_s(p)\big)\geq c_3s^\theta$ for $p\in\G$ and $0<s<1$. 
	
	Define 
	\[
	\begin{aligned}
	&U_{p,s,0}=\bigcup_{w\in \hat{W}_{p,s,0}} F_w\G \text{ with }\hat{W}_{p,s,0}=\{w\in \hat{W}_{s\rho^2}:p\in F_w\G\},\\ &U_{p,s,1}=\bigcup_{w\in \hat{W}_{p,s,1}} F_w\G \text{ with }\hat{W}_{p,s,1}=\{w\in \hat{W}_{s\rho^2}:F_w\G\cap U_{p,s,0}\neq \emptyset\}.
	\end{aligned}
	\]
	Clearly, we have $U_{p,s,0}\subset \mathring{U}_{p,s,1}\subset U_{p,s,1}\subset B_{s}(p)$. Since $(\mcE,\mcF)$ is regular, there exists $f_{p,s}\in \mcF$ so that $f_{p,s}|_{\mathring{U}^c_{p,s,1}}=0$ and $f_{p,s}|_{U_{p,s,0}}=1$. 
	
	As $\G$ satisfies the finite type property, there exists a finite class $\{(p_i,s_i)\}_{i=1}^N$ such that for any $p\in \G$ and $0<s<1$, there exists $1\leq i\leq N$ and an affine map $\psi$ such that $\psi:U_{p,s,l}\to U_{p_i,s_i,l}$ for $l=0,1$, which maps cells corresponding to $\hat W_{p,s,l}$ to those corresponding to $\hat W_{p_i,s_i,l}$. In addition, we can require that $\psi$ maps the boundary of $U_{p,s,l}$ to the boundary of $U_{p_i,s_i,l}$, which only depend on how the outside cells of approximately same size intersect $U_{p_i,s_i,1}$. Thus, we can assume that    
	\[f_{p,s}(q)=\begin{cases}
	f_{p_i,s_i}\circ\psi (q),&\text{ if }q\in U_{p,s,1},\\
	0,&\text{ if }q\in U^c_{p,s,1}.
	\end{cases}\]
	By a similar observation as in Lemma \ref{lemma54}, there exists $m\in \mathbb{Z}$ such that
	\[\mcD^{(n)}(f_{p_i,s_i})\leq \rho_\psi^{\theta}\mcD^{(n+m)}(f_{p,s}),\]
	where $\rho_\psi$ is the similarity ratio of $\psi$. So we have $\mcE(f_{p,s})=\rho_{\psi}^{-\theta}\mcE(f_{p_i,s_i})\leq c_3^{-1}s^{-\theta}$ for some constant $c_3$ independent of $p,s,i$. Since $f_{p,s}|_{B_s^c(p)}=0$ and $f_{p,s}(p)=1$, we get the estimate $R(p,B_s^c(p))\geq c_3s^\theta$.  
	
	Finally, the estimates $R(p,q)\geq c_1d(p,q)^\theta$ follows from the fact that $R(p,B_s^c(p))\geq c_3s^\theta$, and $R\big(p,B_s^c(p)\big)\leq c_4s^\theta$ follows from the fact that $R(p,q)\leq c_2d(p,q)^\theta$.
\end{proof}

By the resistance metric estimate in Proposition \ref{prop63}, the Ahlfors regularity of the measure $\mu_H$ and the resulted estimates from the Nash inequality, there exist a lower bound estimate $p(t,x,y)\geq c_3t^{-d_S/2}$ and an estimate of the hitting time $c_4s^{\theta+d_H}\leq \mathbb{E}^x\tau(x,s)\leq c_5s^{\theta+d_H}$, where $\tau(x,s)=\inf\{t\geq 0:X_t\notin B_s(x)\}$. See \cite{B} for a proof. Finally, by Theorem 3.1.1 of Barlow's book \cite{B} or by following \cite{HK}, we can finally find that our diffusion is a fractional diffusion.

\begin{theorem}\label{thm64}
	The Hunt process $X=(\mathbb{P}^x,x\in\G,X_t,t\geq 0)$ associated with the form $(\mcE,\mcF)$ on $L^2(\G,d_H)$ is a fractional diffusion, with $\beta=\theta+d_H$, in the sense of Definition \ref{def61}.
\end{theorem}

\bibliographystyle{amsplain}

\end{document}